\documentclass[runningheads,12pt]{llncs}
\usepackage{graphicx}
\usepackage{mathptmx}      
\usepackage{authblk}
\usepackage{enumitem}
\usepackage{latexsym}
\usepackage{amsmath}
\usepackage{amssymb}
\usepackage{amscd}
\usepackage[all]{xy}
\usepackage{mathrsfs}
\usepackage{bm}
\usepackage{cite}
\usepackage{hyperref}
\hypersetup{
setpagesize=false,
 bookmarksnumbered=true,%
 bookmarksopen=true,%
 colorlinks=true,%
 linkcolor=black,
 citecolor=red,
}
\DeclareMathOperator*{\essinf}{ess\,inf}
\begin{document}

\title{Composition operators on reproducing kernel Hilbert spaces with  analytic positive definite functions
}

\titlerunning{Composition operators on RKHS with analytic positive definite functions}        

\author{
Masahiro Ikeda\inst{1,2}  \footnote{\email{masahiro.ikeda@riken.jp}} 
\and Isao Ishikawa\inst{1,3} \footnote{\email{ishikawa.isao.zx@ehime-u.ac.jp}} 
\and  Yoshihiro Sawano\inst{1,4} \footnote{\email{yoshihiro-sawano@celery.ocn.ne.jp}}
}

\authorrunning{ }

\institute{    Center for Advanced Intelligence Project, RIKEN, Tokyo, Japan \\
            \and 
            School of Fundamental Science and Technology, Keio University, Kanagawa, Japan\\
            \and
              Center for Data Science, Ehime University, Ehime, Japan\\
            \and
              Department of Mathematics, Chuo University, Tokyo, Japan\\
}


\maketitle

\begin{abstract}
In this paper, we specify
what functions induce the bounded composition operators 
on a reproducing kernel Hilbert space (RKHS) associated with an analytic positive definite function defined on $\mathbb{R}^d$.
We prove that only affine transforms can do so in a certain large class of RKHS.
Our result covers not only the Paley-Wiener space on the real line, studied in previous works, but also much more general RKHSs corresponding to analytic positive definite functions, where existing methods do not work. 
Our method only relies on intrinsic properties of the RKHSs, and we establish a connection between the behavior of composition operators and asymptotic properties of the greatest zeros of orthogonal polynomials on a weighted $L^2$-space on the real line.
We also investigate the compactness of the composition operators and show that any bounded composition operators cannot be compact in our situation.
\end{abstract}

\section{Introduction}
In this paper, we establish that the composition operator generated by a map in the Euclidean space ${\mathbb R}^d$ enjoys the boundedness property in a reproducing kernel Hilbert space (RKHS for short) only if the map is affine 
when the reproducing kernel is an analytic positive definite function with some conditions.
In addition, we will characterize affine maps which induce bounded composition operators.
We summarize several basic notation at the end of this section.

Recall the definition of composition operators. Let $f:E\rightarrow E'$ be a map
from a set $E$ to a set $E'$, 
and let $V$ and $W$ be function spaces on $E$ and $E'$, respectively. 
The {\em composition operator} $C_f:f \mapsto h \circ f$ is the linear operator
$(C_f,\mathcal{D}(C_f))$
from $W$ to $V$ 
whose domain is $\mathcal{D}(C_f)=\{h \in W : h\circ f\in V\}$.
Since a composition operator is always a closed operator,
it is worth remarking that the preserving property, that is, 
$\mathcal{D}(C_f)=W$ implies the boundedness of $C_f$ if $V$ and $W$ are sufficiently good topological linear spaces, for example, Banach spaces.

We also recall the notion of RKHSs. 
A function $k$ defined on the cross product of $E \times E$, 
where $E$ is a set, is said to be positive definite if for any arbitrary function $X:E \to {\mathbb C}$ and for any finite subset $F$ of $E$, 
$ \sum\limits_{p,q \in F} \overline{X(p)} X(q) k(p,q) \geq 0$. 
A fundamental theorem in the theory of RKHSs is that such a function $k$ generates a unique reproducing kernel Hilbert space $H_k$. See \cite{SS}, for example.

Now, let us state our main result. We will adopt the 
following
definition of the {\it Fourier transform}:
\begin{eqnarray*}
\widehat{h}(\xi)=\mathcal{F}[h](\xi)
:=
\int_{{\mathbb R}^d} h(x){\rm e}^{-2\pi{\rm i} x \cdot \xi} 
{\rm d}x. \quad
\end{eqnarray*}
Let $0\neq w \in L^1\cap L^{\infty} (\subset L^2)$ and assume $w \ge 0$ almost everywhere.
Then, $\widehat{w}$ is a {\em positive definite function}, namely, 
\[k(x,y) :=\widehat{w}(x-y)\]
becomes a positive definite kernel, and thus $k$ determine a RKHS $H_k$ (see Section \ref{s2} for more details).
The RKHS $H_k$ is realized in the space of continuous and square integrable functions on $\mathbb{R}^d$
whose Fourier transform vanishes almost everywhere
on $\{w=0\}$, and its norm is given by
\[
\|f\|_{H_k}
=
\sqrt{
\int_{{\mathbb R}^d} |\widehat{f}(\xi)|^2 w(\xi)^{-1} d\xi
}
=\big\| \widehat{f} \big\|_{L^2(w^{-1})}
\quad (f \in H_k).
\]    

It is
worth 
considering the case of $d=1$ and $w = \mathbf{1}_{[-1/2, 1/2]}$.
In this case, the RKHS constructed by $w$ is called the Paley-Wiener space.
This space is composed of functions whose fourier transforms are supported on $[-1/2, 1/2]$.
As {\em warping operators}, 
many researchers have been studying
the condition for which the $\mathcal{D}(C_f)$ is equal to $H_k$, which is equivalent to the boundedness of $C_f$, 
in the field of signal processing, a brunch of engineering \cite{CC90, CC93, Azi02, ACCM, XZ92, CCG}.
In \cite{Lev12}, a general dimensional case is treated although their setting slightly differs 
from ours.
We also note that composition operators have recently attracted researchers in a field of data science and machine learning\cite{HIIK, IFIHK, Kawahara}, and mathematical properties of composition operators are quite important to provide their theoretical guarantee.

For $z \in \mathbb{C}^d$, we denote by $m_z$ the pointwise multiplication operator on $L^2(w)$: 
\[m_z[h](\xi):= {\rm e}^{z^\top \xi}h(\xi),\]
where $\top$ stands for the transpose of vectors or matrices.
For each $n \in {\mathbb N}$, the space $P_n\subset\mathbb{C}[\xi_1,\xi_2,\ldots,\xi_d]$ stands for the linear space of all polynomials of (total) degree at most $n$. 
In the case of  $|\xi|^{2n} w(\xi) \in L^1$, we regard $P_n \subset L^2(w)$, and for any $y \in \mathbb{R}^d$, we define
\begin{align*}
     \mathcal{E}_n^+(y; w) &:= \sqrt[n]{ \sup\limits_{P \in P_n \setminus \{0\}} \frac{\|m_y P\|_{L^2(w)}}{\|P\|_{L^2(w)}}}  = \| m_y|_{P_n} \|^{1/n},\\
     \mathcal{E}_n^-(y; w) &:= \sqrt[n]{ \sup\limits_{P \in P_n \setminus \{0\}} \frac{\|P\|_{L^2(w)}}{\|m_y P\|_{L^2(w)}}}  = \| m_y|^{-1}_{P_n} \|^{1/n}.\\
\end{align*}
We define
\[
{\mathcal G}(w)
:=
\left\{A \in {\rm GL}_d({\mathbb R})\,:\,
w(A^\top \xi) \ge \lambda w(\xi) 
\mbox{ a.e. $\xi$ for some $\lambda>0$}
\right\}.
\]
We will check 
that $\mathcal{G}(w)$ coincides with the linear maps inducing a bounded composition operator on $H_k$ (see Proposition \ref{prop: boundedness for affine maps}).

Under certain assumptions on $w$, the boundedness of composition operators force the original maps to be affine as our main result below shows: 

\begin{theorem}
\label{main thm}
Let $w\in L^1 \cap L^\infty \setminus \{0\}$ be non-negative almost everywhere, and let $$k(x,y):=\widehat{w}(x-y).$$ 
We impose the following three assumptions on $w$:
\begin{enumerate}[label={\rm (\Alph*)}]
\item for any $a>0$, there exists $c_a>0$ such that $w(\xi)\le c_a{\rm e}^{-a|\xi|}$
for almost all $\xi \in {\mathbb R}^n$, 
\label{main thm assumption 1}
\item there exists $B>0$ such that for any $y\in \mathbb{R}^d$,
\[ \limsup_{n\rightarrow \infty} \mathcal{E}_n^-(y; w),~ \limsup_{n\rightarrow\infty} \mathcal{E}^+_n(y; w) < B, \]
\label{main thm assumption 2}
\item $\mathcal{G}(w)$ spans ${\rm M}_d(\mathbb{R})$, that is, $\langle \mathcal{G}(w)\rangle_{\mathbb{R}} = {\rm M}_d(\mathbb{R})$.
\label{main thm assumption 4}
\end{enumerate}
Then, for any open set $U\subset\mathbb{R}^d$ and any map $f:U\rightarrow\mathbb{R}^d$, the map $f$ is a restriction of an affine map of the form, 
\[f(x)=Ax+b\] 
with $A\in\mathcal{G}(w)$ and $b\in\mathbb{R}^d$ if and only 
if $\mathcal{D}(C_f)=H_k$, and thus the composition operator $C_f:H_k\rightarrow H_{k|_{U^2}}$ is bounded. \label{thm:191007-1}
\end{theorem}
We will prove Theorem \ref{main thm} as a corollary of Theorems \ref{thm: affineness} and \ref{thm: analytic continuation}.
Our method is based on an intrinsic structure of $L^2(w)$, and thus can treat a quite general class of $w$ on higher dimensional Euclidean space $\mathbb{R}^d$ ($d \ge 1$).
For some special
cases, we have a more concrete result as follows:
\begin{theorem}
\label{main thm 2}
Let $w \in L^1 \cap L^\infty \setminus \{0\}$ be 
a nonnegative spherical function.
Assume that there exists
a locally $L^1$-function $Q:[0,\infty) \rightarrow \mathbb{R}$ such that $w(\xi) = {\rm e}^{-Q(|\xi|)}$.
We further assume that there exists $c \ge 0$ such that  $Q(t) + ct$ is non-decreasing for sufficiently large $t \ge 0 $ and
that $Q(t+R)-Q(t) \rightarrow \infty$ as $t\rightarrow \infty$ for some $R>0$.
Then, the function $w$ satisfies Assumptions
$\rm \ref{main thm assumption 1}$,
$\rm \ref{main thm assumption 2}$, and 
$\rm \ref{main thm assumption 4}$.
Thus, 
if a map induces
bounded composition operators on the RKHS $H_k$ associated to $k(x,y)=\widehat{w}(x-y)$,
then the original map is affine.
\end{theorem}
For example, $w(\xi) = |\xi|^{-|\xi|}$, $\Gamma(|\xi|)^{-1}$, $|\xi|^\alpha {\rm e}^{-|\xi|\beta(|\xi|)}$ , where $\alpha \ge 0$ and $\beta: [0,\infty)\rightarrow \mathbb{R}$ is an non-decreasing function going to $\infty$, satisfy the conditions 
in Theorem \ref{main thm 2}.
We actually obtain a more general result (Theorem \ref{thm: criterion for spherical pdf}), and Theorem \ref{main thm 2} above is a corollary 
(see Corollary \ref{cor: super easy criterion}) of Theorem \ref{thm: criterion for spherical pdf}.
We also obtain a similar result in the case where $w$ is a tensor product of even functions on $\mathbb{R}$ (see Section \ref{sec: tensor prod} for details).

Our results cover the previous works 
\cite{CC93, CCG}, namely,
 we see that Assumptions (A)--(C)
 in Theorem \ref{main thm} hold 
if $d=1$ and $w$ is compactly supported, 
since $m_z$ is a bounded linear isomorphism on $L^2(w)$ and $1 \in \mathcal{G}(w)$.
Needless to say,
 $\mathbf{1}_{[-\pi,\pi]}$ satisfies
 Assumptions (A)--(C) in Theorem \ref{main thm}.
 Thus, only affine maps can induce bounded linear operators on 
the Paley--Wiener space in this case.
 
Furthermore, our result also provides a non-trivial improvement even in a one-dimensional case.
In fact, in \cite{CCG, CC93}, their proof is based on finiteness of the order of the entire function $\widehat{w}$.
In their situation, the RKHS is composed of entire functions of order at most $1$, and they directly use P\'olya's theorem \cite{Po26} with some careful analysis, in other words, finiteness of the order of $\widehat{w}$ is crucial in their proof.
In contrast, our method does not need finiteness of the order, and by means of Theorem \ref{main thm 2}, we actually find an example of RKHS containing entire functions of infinite order, but only affine maps can induce bounded composition operators (see Section \ref{subsec: remark} for details) as in the following corollary:
\begin{corollary}
Assume $d=1$. Let
\[w(\xi) = \sum\limits_{n=-\infty}^\infty \frac{\mathbf{1}_{[-1/2+n, 1/2 +n)}(\xi)}{|n|!}.\]
Then, if a map induces
bounded composition operators on the RKHS $H_k$ associated to $k(x,y)=\widehat{w}(x-y)$,
then the original map is affine, but the RKHS $H_k$ contains entire functions of infinite order.
\end{corollary}
In addition, in previous works, they always impose some good properties, such as entireness, on the original map.
On the other hand, we do not need to assume the map $f$ is entire in advance as we prove the boundedness of $C_f$ automatically induces the holomorphy and continuity of $f$ (Theorem \ref{thm: analytic continuation}).

The compactness of composition operators is
an
important problem as well.
However, unfortunately,
compactness fails as our next theorem shows.
\begin{corollary}\label{cor:191118-1}
Using the same notation as Theorem \ref{main thm}, let $w$ satisfy 
Assumptions $\rm \ref{main thm assumption 1}$--$\rm \ref{main thm assumption 4}$. Then, no composition operators $C_f$ {\em can} be compact.
\end{corollary}
We can rephrase Corollary \ref{cor:191118-1} in Proposition \ref{non-compactness} on the basis of Theorem \ref{main thm}, namely, thanks to Theorem \ref{main thm}, it suffices to prove that affine maps cannot induce compact composition operators. See Section \ref{subsec: criteria affine} for the details.

Under Assumption \ref{main thm assumption 1}, 
$H_k$ is composed of entire functions on $\mathbb{C}^d$, thus our study is also
located as a study of composition operators on entire functions.
This topic has been extensively studied, for example, as in \cite{CMS, AA17, DKL17, SS17} and so on.
However, the function space $H_k$ in our study 
is quite different from those in these previous literature.
As a result, the behavior of composition operators dramatically changes compared with them, for example, composition operators therein can become compact linear operators, but does not in our situation as in the above corollary. 

Let us explain the outline of the proof of Theorem \ref{main thm} and \ref{main thm 2}. 
Regarding Theorem \ref{main thm}, the ``if'' part is not so hard, and we prove it in Section \ref{subsec: criteria affine}. 
The harder part of the proof 
is the ``only if'' part, and is obtained as a corollary of Theorems \ref{thm: affineness} and \ref{thm: analytic continuation} in Section \ref{s3}. 
Theorem \ref{thm: affineness} states that the ``only if'' part of Theorem \ref{main thm} holds under Assumptions \ref{main thm assumption 1}, \ref{main thm assumption 2}, and \ref{main thm assumption 4}, and assuming the existence of the holomorphic map $F:\mathbb{C}^d\rightarrow\mathbb{C}^d$ with $F|_U=f$. 
Since there exists a natural isomorphism $\Psi_w: H_k \cong L^2(w)$ (see Corollary \ref{cor: natural isom}), the study of composition operators is reduces to that of the corresponding operators on $L^2(w)$. 
Thus, by considering the action of the corresponding operators on spaces of polynomials of various degrees in $L^2(w)$, the boundedness of the composition operators enables us to control the derivative of the holomorphic map $F$.
Theorem \ref{thm: analytic continuation} deduces that under Assumption \ref{main thm assumption 1} and $\mathcal{D}(C_f)=H_k$,
there exists a holomorphic map $F:\mathbb{C}^d\rightarrow\mathbb{C}^d$ with $F|_U=f$. 
The proof involves an explicit construction of the analytic continuation of $f$ in terms of the boundedness of $C_f$. 
Here, we emphasize that $f$ is originally a mere map defined on an open subset $U\subset\mathbb{R}^d$, not the whole space, but we prove that the boundedness of $C_f$ shows the map $f$ is a restriction of holomorphic map $F$ defined on $\mathbb{C}^d$. 
As for Theorem \ref{main thm 2}, the hardest part in the proof is to 
verify Assumption \ref{main thm assumption 2}.
We show that
the left hand side in Assumption \ref{main thm assumption 2} 
is closely related to an asymptotic behavior of greatest zeros 
of orthogonal polynomials in a weighted $L^2$-space on $\mathbb{R}$.
We improve Freud's methods
in his works \cite{Fre74-2, Fre86}
on asymptotic behaviors of orthogonal polynomials, 
and check Assumption \ref{main thm assumption 2}.
We include the details in Section \ref{section4}.
\paragraph{\bf Notation:}
In this paper, we always work in  the $d$-dimensional Euclidean space.
For $r \ge 0$, we denote by $C^r$ the space 
of $C^r$ functions on $\mathbb{R}^d$.
For a non-negative measurable function  $w$ on $\mathbb{R}^d$, and $p\in[1,\infty)$, we denote by $L^p(w)$ the space composed of the equivalent classes of measurable functions $h$ vanishing on $\{ w = \infty \}$ such that $\int_{\mathbb{R}^d} |h(x)|^p w(x) {\rm d}x < \infty$.
In the case of $w \equiv 1$, we abbreviate $L^p(1)$ to $L^p$.
We also denote by $L^\infty$ the space of essentially bounded measurable functions on $\mathbb{R}^d$.
For any measurable set $A$ in an Euclidean space, we denote by $\mathbf{1}_A$ the characteristic function supported on $A$, and by $|A|$ the volume of $A$ with respect to the Lebesgue measure. 
We denote by ${\rm M}_d(\mathbb{R})$ (resp. ${\rm GL}_d(\mathbb{R})$) the set of square real matrices of size $d$ (resp. regular real matrices of size $d$).

\paragraph{\bf Acknowledgement} We would like thank 
Professor Takeshi Katsura and Dr. Fuyuta Komura
at Keio University for valuable discussion and comments. 
Thanks to their remarks, we could show that $\phi_k^{-1}$ is continuous 
in Proposition \ref{prop: phi homeo}. 
This work was supported by a JST CREST Grant, 
Number JPMJCR1913, 
Japan, and JST ACTX Grant, Number JPMJAX2004, Japan.
The third author is supported 
by Grant-in-Aid for Scientific Research (C) (19K03546), 
the Japan Society for the Promotion of Science.
The third author
is supported by the Japan Society 
for People’s Friendship University of Russia.
\section{Preliminaries}
\label{s2}
In this section, we review the notion of reproducing kernel Hilbert spaces associated with positive definite functions and composition operators and then show some of their basic properties.

\subsection{Reproducing kernel Hilbert spaces with positive definite functions}
Let $E$ be an arbitrary abstract (non-void) set and $k:E\times E\rightarrow\mathbb{C}$ be a map. Denote by $\mathbb{C}^E$ the linear space of all maps from $E$ to $\mathbb{C}$. A {\em reproducing kernel Hilbert space} 
{\rm (}RKHS for short{\rm)} with respect to $k$ is a Hilbert space $H_k \subset \mathbb{C}^E$ satisfying the following two conditions:
\begin{enumerate}
\item For any $x\in E$, the map $k_x:= k(x,\cdot)$ is an element of $H_k$.
\item For any $x\in E$ and $h\in H_k$, we have $\langle h, k_x\rangle_{H_k}=h(x)$.

\end{enumerate}
If such $H_k$ exists, $H_k$ is unique as a set and we call $k$ a {\em positive definite kernel}. We note that if $k$ is a positive definite function in the sense described in Introduction, there exists a unique Hilbert space $H_k \subset \mathbb{C}^E$ satisfying the above two conditions 
(see, \cite{SS} for more detail).
The second condition is known as the reproducing property 
of the RKHS $H_k$. 
Here, we define the {\em feature map} by
\[\phi_k:E\rightarrow H_k;~ x\mapsto k_x.\] 
For any subset $F\subset E$, we define a closed subspace of $H_k$ by
\[ H_{k, F} = \overline{\langle \phi_k(F) \rangle_\mathbb{C}}, \]
which is the closure of the linear subspace generated by the set $\phi_k(F)$. Accordingly, we see that $H_{k,F}$ is isomorphic to $H_{k|_{F^2}}$:
\begin{proposition}
\label{prop: restriction of kernel}
Let $k$ be a positive definite kernel on $E$, and let $F\subset E$. Then, the restriction map $\mathbb{C}^{E}\rightarrow\mathbb{C}^F$ induces an isomorphism $r_{k,F}:H_{k,F}\rightarrow H_{k|_{F^2}}$.
\end{proposition}
\begin{proof}
For any $x\in F$, the restriction map allocates $\phi_{k|_{F^2}}(x)$ to $\phi_k(x)$; thus, it induces the isomorphism $r_{k,F}$ between the Hilbert spaces. 
\end{proof}

\begin{definition}
A map $u:{\mathbb R}^d\rightarrow{\mathbb C}$ is a {\em positive definite function} if $k(x,y):=u(x-y)$ is a positive definite kernel. 
We call $H_k$ the RKHS associated with $u$.
\end{definition}
Thanks to Bochner's theorem \cite[p. 148]{Katznelson-text-04}, a positive definite function on $\mathbb{R}^d$ can be realized as a Fourier transform of a finite Borel measure; namely, we have the following proposition.
\begin{proposition}
Let $u$ be a non-zero ${\mathbb C}$-valued continuous function on ${\mathbb R}^d$. Then, $u$ is a positive definite function if and only if there exists a finite Borel measure $\mu$ on $\mathbb{R}^d$ such that 
\[u(x)
=\widehat{\mu}(x)
:=\int_{{\mathbb R}^d}{\rm e}^{-2\pi ix\cdot \xi}\,d\mu(\xi).\]
\end{proposition}

We shall now prove a proposition for the feature map for an RKHS associated with a positive definite function $u$:
\begin{proposition}
\label{prop: phi homeo}
Let $u$ be a non-zero positive definite function. Assume that $u$ is continuous and $u(x)\rightarrow 0$ as $|x|\rightarrow \infty$. Write $k(x,y):=u(x-y)$ as above. Then the feature map $\phi_k:\mathbb{R}^d\rightarrow H_k$ is injective and continuous. Moreover, the inverse $\phi_k^{-1}:\phi_k\left(\mathbb{R}^d\right)\rightarrow\mathbb{R}^d$ is also continuous.
\end{proposition}
\begin{proof}
Since $|\!|\phi_k(x)-\phi_k(y)|\!|_{H_k}^2 = 2u(0)-2{\rm Re}(u(x-y))$, 
the continuity of $u$ implies that of $\phi_k$. We will prove the injectivity of $\phi_k$. Suppose
to the contrary 
that there exists $a,b\in\mathbb{R}^d$ with $a\neq b$ such that $\phi_k(a) = \phi_k(b)$. Then, for any $x\in \mathbb{R}^d$, 
\[u(x)=\langle \phi_k(a+x), \phi_k(a)\rangle_{H_k}=\langle \phi_k(a+x), \phi_k(b)\rangle_{H_k}=u(x+(a-b)),\] 
which contradicts the assumption that $u$ vanishes at infinity. Thus, the feature map is injective. Now, we prove the continuity of $\phi_k^{-1}:\phi_k(\mathbb{R}^d)\rightarrow\mathbb{R}^d$. Suppose that there exists a sequence $\{\phi_k(x_n)\}_{n\ge1}$ such that $\phi_k(x_n)\rightarrow \phi_k(a)$ for some $a\in\mathbb{R}^d$ but that $\{x_n\}_{n\ge0}$ does not converge to $a$. Since $\phi_k$ is injective and continuous, any convergent subsequence of $\{x_n\}_{n\ge0}$ converges to $a$; thus, we may assume $|x_n|\rightarrow\infty$ as $n\rightarrow\infty$. For any $x\in \mathbb{R}^d$, 
\begin{align*}
u(x)&=\langle \phi_k(x+a), \lim\limits_{n\rightarrow\infty}\phi_k(x_n)\rangle_{H_k}\\
&=\lim\limits_{n\rightarrow\infty}\langle \phi_k(x+a), \phi_k(x_n)\rangle_{H_k}\\
&=\lim\limits_{n\rightarrow\infty}u(x+a-x_n)\\
&=0.
\end{align*}
Since $u$ is not a constant function, this is a contradiction. Thus, $\phi_k^{-1}$ is continuous. 
\end{proof}
By the Riemann-Lebesgue theorem, we have the following corollary:
\begin{corollary}
Assume $H_k$ is an RKHS associated with a positive definite function in the form of $\widehat{w}$ 
for some non-negative $w\in L^1 \cap L^\infty \setminus \{0\}$.
Then the feature map $\phi_k$ is a homeomorphism from $\mathbb{R}^d$ onto $\phi_k(\mathbb{R}^d)$.
\end{corollary}
In this paper, we will only consider the RKHSs associated with a positive definite function in the form of $\widehat{w}$ for some $w\in L^1 \cap L^\infty \setminus \{0\}$.

We have an explicit description of the RKHS $H_k$ above:
\begin{proposition}
\label{prop: explicit expression}
Let $w \in L^1 \cap L^\infty \setminus \{0\}$, 
and let $k(x,y)=\widehat{w}(x-y)$. Then, 
\[H_k = \left\{h\in C^0\cap L^2: \widehat{h}\in L^2(w^{-1})\right\}\] 
and its inner product is given by $\langle g,h\rangle_{H_k}=\int_{\mathbb{R}^d}\widehat{g}(\xi) \overline{\widehat{h}(\xi)}w(\xi)^{-1}\,d\xi$.
In particular, the Fourier transform of any element in $H_k$ is in $L^1$.
\end{proposition}
\begin{proof}
Since $L^2 \supset L^2(w^{-1})$, we define $V:=\mathcal{F}^{-1}\left(L^2(w^{-1})\right)$ be a Hilbert space with inner product $\langle g,h\rangle_V:=\big\langle \widehat{g},\widehat{h}\big\rangle_{L^2(w^{-1})}$.
By direct computation, we see that $k_x \in V$ and for any $h \in V$, $\langle h, k_x \rangle_V = h(x)$ almost all $x \in \mathbb{R}^d$.
We will show $V$ is regarded as a subspace of $C^0$; namely, any element $h\in V$ has a continuous representative.
For $x\in\mathbb{R}^d$, let $\widetilde{h}(x):= \langle h, k_x\rangle_V$. Since $\widetilde{h}(x)=h(x)$ for almost all $x\in\mathbb{R}^d$, it suffices to show that the map $\widetilde{h}$ is continuous. In fact, it 
is a consequence of the following inequality:
\[|\widetilde{h}(x)-\widetilde{h}(y)|\le 
\|h\|_V\cdot(k(x,x)+k(y,y)-k(x,y)-k(y,x))^{1/2} \quad(x,y \in {\mathbb R}^d).
\]
This inequality can be deduced from Schwartz's inequality.
Thus, we may regard $V \subset C^0 \subset \mathbb{C}^{\mathbb{R}^d}$.
By the uniqueness of the RKHS, we have $H_k = V$. 
\end{proof}
We define $\Psi_w: H_k \rightarrow L^2(w)$ by 
\[\Psi_w(h)(\xi):=w(\xi)^{-1}\widehat{h}(\xi).\]
Then, we immediately obtain the following corollary:
\begin{corollary}
\label{cor: natural isom}
Let $w \in L^1 \cap L^\infty \setminus \{0\}$ 
be a non-negative measurable function.
Write $k(x,y)=\widehat{w}(x-y)$. 
Then, $\Psi_w$ is an isomorphism from $H_K$ to $ L^2(w)$.
\end{corollary}

\subsection{Properties of RKHSs for positive definite functions with a certain decay condition}
Let $w \in L^1 \cap L^\infty \setminus \{0\}$ 
be a non-negative measurable function.
We define the following decay condition: For a positive integer $n>0$ and positive real number $a\ge0$, we define a function $u$ that satisfies $({\rm DC})_{n,a}^d$ if for any $\varepsilon >0$, there exists $L_\varepsilon>0$ such that
\begin{align*}
({\rm DC})_{n,a}^d&\hspace{25pt} w(\xi)
\le L_\varepsilon
(1+|\xi|^{2n+d+\varepsilon})^{-1}
{\rm e}^{-4\pi a|\xi|}~~{\rm a.e.}~\xi.
\end{align*}

For $a>0$, we define 
\[\mathfrak{X}_a^d:=\left\{z=(z_i)_{i=1}^d\in\mathbb{C}^d : |{\rm Im}(z)|<a\right\}.
\]
We also define
\[\mathfrak{X}_0^d:=\left\{z=(z_i)_{i=1}^d\in\mathbb{C}^d : {\rm Im}(z)=0
\right\}={\mathbb R}^d.
\]
By virtue of Proposition \ref{prop: explicit expression}, if $u$ satisfies $({\rm DC})_{n,a}^d$ for some $a>0$ (resp. $a=0$), then any element of $H_k$ is holomorphic on $\mathfrak{X}_a^d$ (resp. $C^n$ on $\mathfrak{X}_0^d$). As a result, we have the following proposition:

\begin{proposition}
\label{prop: local global}
Let $w \in L^1 \cap L^\infty \setminus \{0\}$ be a non-negative function 
satisfying the condition $({\rm DC})_{n,a}^d$ for some integer $n>0$ and $a>0$, and let $k(x,y)=\widehat{w}(x-y)$.
Then, we have
\[H_{k,U}=H_k\]
for any open subset $U\subset \mathbb{R}^d$. 
\end{proposition}

\begin{proof}
It suffices to prove that $H_{k,U}^{\perp}=\{0\}$. Take an arbitrary $g\in H_{k,U}^\perp$. Then, we see that for any $x\in U$,
\[g(x)=\langle \phi_k(x), g\rangle_{H_k}=0.\] 
Since $g$ is an analytic function on $\mathbb{R}^d$, we have $g=0$. Therefore, $H_{k,U}^\perp=\{0\}$. 
\end{proof}

Under the condition $({\rm DC})_{n,a}^d$, for each $z\in \mathfrak{X}_a^d$, 
we define ${\bf e}_z\in L^2(w)$ by
\[{\bf e}_z(\xi):={\rm e}^{-2\pi i z^\top\xi}, \]
and we define the map,
\[\varphi:\mathfrak{X}_a^d\longrightarrow L^2(w); z\mapsto {\bf e}_z.\]
We should remark that in the case of $a>0$, for any $z\in\mathfrak{X}_a^d$, $\varphi(z)=\Psi_w(\phi_k(z))$, where $\phi_k(z)(x)$ is defined as the evaluation of the analytic continuation of $u$ at $x-z$ (note that $u$ is originally defined on $\mathbb{R}^d$). 
Accordingly, we have the following proposition:
\begin{proposition}
\label{prop: analyticity of feature maps}
Let $w \in L^1 \cap L^\infty \setminus \{0\}$ be 
a non-negative function satisfying the condition 
$({\rm DC})_{n,a}^d$ 
for some integer $n>0$ and $a>0$ $($resp. $a=0$$)$.
Then the map $\varphi$ is holomorphic $($resp. differentiable$)$ in $\mathfrak{X}_a^d$ in the sense that for any $z=(z_j)_{j=1}^d\in \mathfrak{X}_a^d$, the limit 
\[\partial_{z_j}\varphi(z):=\lim\limits_{\substack{\varepsilon\rightarrow 0
}
}
\varepsilon^{-1}\left(\varphi(z+\varepsilon e_j)-\varphi(z)\right)\]
exists. Here, $e_j:=(\underset{1}{0},\ldots,0,\underset{j}{1},0,\dots,\underset{d}{0})\in \mathbb{C}^d$ denotes the $j$-th elementary vector. Moreover, for any $d$-variable polynomial (resp. $d$-variable polynomial of degree smaller than or equal to $n$) $q\in\mathbb{C}[\xi_1,\dots,\xi_d]$, we have
\begin{align}
\label{eq: derivative and polynomial}
\left[q\left(\frac{i\partial_{z_1}}{2\pi},\dots,\frac{i\partial_{z_d}}{2\pi}\right)\varphi\right](z)=q{\bf e}_z.
\end{align}
\end{proposition}
\begin{proof}
For any positive number $\varepsilon>0$,
any non-negative integer $n\ge0$, $j=1,\dots,d$, and any function $\psi:\mathfrak{X}_a^d\rightarrow L^2(w)$, we define
\[\left(\Delta^{(n)}_{j,\varepsilon}\psi\right)(z)
:=(\varepsilon^nn!)^{-1}\sum\limits_{r=0}^n(-1)^{n-r}\left(\begin{array}{c}n\\r\end{array} \right)\psi(z+r\varepsilon e_j).\]
Moreover, for $\bm{n}=(n_1,\dots,n_d)$, we define
\[\Delta_\varepsilon^{(\bm{n})}=\Delta_{1,\varepsilon}^{(n_1)}\cdots\Delta_{d,\varepsilon}^{(n_d)}.\]
It suffices to show that 
\begin{align}\lim\limits_{\varepsilon\rightarrow 0}\left(\frac{i}{2\pi}\right)^{|\bm{n}|}\left(\Delta_\varepsilon^{(\bm{n})}\varphi\right)(z)
=\xi_1^{n_1}\cdots\xi_d^{n_d}{\bf e}_z \label{eq: derivative}
\end{align}
for any $\bm{n}$ if $a>0$, or $|\bm{n}|\le n$ if $a=0$. Here, we denote $|\bm{n}|:=\sum\limits_{j}n_j$. By direct computation, we have
\begin{align*}
&\|\text{(left hand side of (\ref{eq: derivative}))} - \text{(right hand side of (\ref{eq: derivative}))}\|_{L^2(w)}^2\\
&=\int_{\mathbb{R}^d}\left|\prod_{j=1}^d\left(\frac{{\rm e}^{-2\pi i\varepsilon\xi_j}-1}{-2\pi i\varepsilon}\right)-\prod_{j=1}^d\xi_j^{n_j}\right|^2{\rm e}^{2\pi {\rm Im}(z)^\top \xi}w(\xi)\,{\rm d}\xi.
\end{align*}
Thus, from the definition of $({\rm DC})_{n,a}^d$, we see that the last integral converges to 0. 
\end{proof} 

\subsection{Composition operators on RKHS}
We give a definition of composition operators for our setting:
\begin{definition}
Let $k$ and $\ell$ be positive definite kernels on sets $E$ and $F$, respectively. For any map $f:E\rightarrow F$, the {\em composition operator} $C_f:H_{\ell}\rightarrow H_k$ is a linear operator defined by $C_f(h):=h\circ f$ of domain
\[\mathcal{D}(C_f)=\left\{h\in H_\ell:h\circ f\in H_k\right\}.\] 
\end{definition}
Since we see that a composition operators is closed, by the closed graph theorem, we have the following proposition:
\begin{proposition}
\label{prop: preserving property and boundedness}
If $\mathcal{D}(C_f)=H_\ell$, the composition operator $C_f$ is a bounded operator.
\end{proposition}
Accordingly, the adjoint of $C_f$ has the following property:
\begin{proposition}
\label{prop: adjoint of composition operator}
Let $k$ and $\ell$ be a positive definite kernel on $E$ and $F$, respectively, and let $f:E\rightarrow F$ be a map such that $\mathcal{D}(C_f)=H_\ell$.
Then, we have
\[C_f^*(\phi_k(x))=\phi_\ell(f(x))
\quad (x \in E).\]
\end{proposition}
\begin{proof}
The proof entails a straightforward computation.  
\end{proof}
Consequently, we have the following corollary:
\begin{corollary}
\label{cor: continuity of f}
Let $E\subset\mathbb{R}^d$, and let $f:E\rightarrow \mathbb{R}^{d}$ be a map.
Suppose that we have a non-zero positive definite function $u$ on $\mathbb{R}^d$.
Write $k(x,y):=u(x-y)$ as above. Assume $u$ is continuous and $u(x)\rightarrow 0$ as $|x|\rightarrow \infty$. 
Then, if $\mathcal{D}(C_f) = H_{k}$, the original map $f$ is continuous.
\end{corollary}
\begin{proof}
From Propositions 
\ref{prop: restriction of kernel}, 
\ref{prop: preserving property and boundedness} 
and 
\ref{prop: adjoint of composition operator}, we see that
\[f = \phi_k^{-1}\circ C_f^*\circ r_{k,E}\circ\phi_k.\]
Since the right-hand side is continuous, so is $f$. 
\end{proof}

At the end of this section, we define another linear operator $K_f$ under the condition $({\rm DC})_{n,a}^d$ for some $a>0$, keeping in mind that $H_k=H_{k,U}$ according to Proposition \ref{prop: local global}.

\begin{definition}
\label{PF operator}
Let $U\subset \mathbb{R}^d$ be an open subset. 
For any map $f:U\rightarrow {\mathbb R}^d$ such that $\mathcal{D}(C_f)=H_k$, define a linear operator $K_f:L^2(w)\rightarrow L^2(w)$ by
\begin{align}
\label{eq: pf operator}
K_f:L^2(w)\overset{\Psi_w^{-1}}{\cong}H_k=H_{k,U}\overset{r_{k,U}}{\cong}H_{k|_{U^2}}\overset{C_f^*}{\longrightarrow}H_k\overset{\Psi_w}{\longrightarrow}L^2(w).
\end{align}
Here, $r_{k,U}$ is the restriction map defined 
in Proposition \ref{prop: restriction of kernel}.
\end{definition}

\section{Main results}
\label{s3}
Here, we prove the main results. We establish the criterion 
of the boundedness of composition operators in the case that the map $f$ is affine. 

\subsection{Boundedness and compactness of composition operators for affine maps}
\label{subsec: criteria affine}
Recall that we defined
\begin{align*}
{\mathcal G}(w)
&=
\bigcup_{\lambda>0}
\left\{A \in {\rm GL}_d({\mathbb R})\,:\,
w(A^\top \xi) \ge \lambda w(\xi)
\mbox{ for almost all $\xi \in {\mathbb R}^d$}
\right\}.
\end{align*}
As the following proposition shows, ${\mathcal G}(u)$ is a natural class.
\begin{proposition}
\label{prop: boundedness for affine maps}
Let $w \in L^1 \cap L^\infty \setminus \{0\}$ be a non-negative function.
Write $k(x,y)=\widehat{w}(x-y)$ for $x,y \in \mathbb{R}^d$. Let $f:\mathbb{R}^d\rightarrow\mathbb{R}^d$ be an affine map, namely, $f(x)=Ax+b$ such that $A\in {\rm M}_d({\mathbb R}^d)$ and $b\in\mathbb{R}^d$.
Then, $A \in {\mathcal G}(w)$ if and only if $\mathcal{D}(C_f)=H_k$ (and $C_f$ is bounded on $H_k$).
\end{proposition}

\begin{proof}
First, we prove the ``if'' part.
Since any element of $H_k$ vanishes at $\infty$ and $C_f$ preserves $H_k$, the matrix $A$ has to
be regular.
Let $h$ be an arbitrary non-negative smooth function with compact support and vanishing in an open set including $\{w=0\}$. We define $g:=\Psi_w^{-1}(h^{1/2}w^{-1})=\mathcal{F}^{-1}(h^{1/2})$. Since $C_f$ is bounded, there exists $L>0$ such that
\[ L|\!|g|\!|_{H_k}^2-|\!|C_fg|\!|_{H_k}^2\ge0.\]
Since 
\begin{align*}
\mathcal{F}(C_fg)(\xi) &= \int_{\mathbb{R}^d} g(Ax + b){\rm e}^{-2\pi i x^\top\xi} {\rm d}x\\
&= |\det A|^{-1}{\rm e}^{2\pi i b^\top\xi} \int_{\mathbb{R}^d} g(x) {\rm e}^{-2\pi i x^\top(A^{-\top}\xi)} {\rm d}x\\
&= |\det A|^{-1}{\rm e}^{2\pi i b^\top\xi} \mathcal{F}(g)(A^{-\top}\xi),
\end{align*}
by Proposition \ref{prop: explicit expression}, we see that
\[\int_{\mathbb{R}^d}h(\xi)\left(|\det{A}|\cdot Lw(\xi)^{-1}-w(A^\top\xi)^{-1}\right)\,{\rm d}\xi\ge0.\]
Since $h$ is arbitrary, we have 
\[w(A^\top\xi)\ge \det{A}^{-1}L^{-1}w(\xi),\]
namely, $A\in \mathcal{G}(w)$.

Now we prove the ``only if'' part. Let $A\in\mathcal{G}(w)$, and let $\lambda>0$ such that $w(A^\top \xi) \ge \lambda w(\xi)$ for almost all $\xi$. As in the same computation as above, for any $g\in H_k$, we have
\[ \lambda^{-1}|\!|g|\!|_{H_k}^2-|\!|C_fg|\!|_{H_k}^2\ge0.\]
Thus, we conclude that $\mathcal{D}(C_f)=H_k$ and that $C_f$ is bounded on $H_k$. 
\end{proof}

We also observe that the composition operators induced by affine maps cannot be compact:
\begin{proposition}
\label{non-compactness}
Let $w \in L^1 \cap L^\infty \setminus \{0\}$ be a non-negative function.
Write $k(x,y)=\widehat{w}(x-y)$ for $x,y \in \mathbb{R}^d$ as before. 
Let $f:\mathbb{R}^d\rightarrow\mathbb{R}^d; x\mapsto Ax+b$ be such that $A \in {\mathcal G}(w)$. 
Then the composition operator $C_f$ cannot be a compact operator
on $H_k$.
\end{proposition}
\begin{proof}
Put $u = \widehat{w}$.
It suffices to show that the operator $K_f$ (Definition \ref{PF operator}) cannot be compact. 
Assume to the contrary that $K_f$ is compact. Fix a sequence $\{x_n\}_{n\ge0}$ such that $\inf\limits_{m,n>R}|x_m-x_n|\rightarrow\infty$ 
as $R\rightarrow\infty$ (for example, $x_n=(n^2,0,\dots,0)$). 
Since $K_f$ is compact, we may assume $\{K_f(\varphi(x_n))\}$ converges to an element $h\in L^2(w)$. Meanwhile, 
$\|K_f\varphi(x_m)-K_f\varphi(x_n)\|_{L^2(w)}=2u(0)-2{\rm Re}(u)(A(x_m-x_n))$. 
Since $u(x)$ converges to 0 as $|x|\rightarrow\infty$ by the Riemann--Lebesgue theorem, 
by taking the limit $m,n\rightarrow\infty$ with $m\neq n$, we find that $u(0)=0$. 
Since $u$ is a positive definite function, $|u(x)| = |\langle k_x, k_0 \rangle_{H_k}| \le u(0)$ by the Cauchy-Schwarz inequality.
Thus, we have $u=0$. This is a contradiction. 
\end{proof}

\subsection{Affineness of holomorphic maps with bounded composition operators}

In this section, 
we prove that maps are affine if they admit an analytic continuation 
and the domains of their composition operators are the whole space $H_k$, namely, the following theorem:
\begin{theorem}
\label{thm: affineness}
Let $w \in L^1 \cap L^\infty \setminus \{0\}$ be a non-negative function,
and let $k(x,y)=\widehat{w}(x-y)$. We impose the following three assumptions on $w$:
\begin{enumerate}[label=(\Alph*)]
\item the function $w$ satisfies $({\rm DC})_{n,a}^d$ for all $a>0$,\label{thm A: decay condition}
\item there exists a constant $B >0$ such that for any $y \in \mathbb{R}^d$,
\[ \limsup_{n\rightarrow \infty} \mathcal{E}_n^-(y; w),~ \limsup_{n\rightarrow\infty} \mathcal{E}^+_n(y; w) < B, \]
\label{thm A: technical condition}
\item $\langle \mathcal{G}(w)\rangle_{\mathbb{R}} 
= {\rm M}_d(\mathbb{R})$.
\label{thm A: generator condition}
\end{enumerate}
Then, for any open set $U\subset{\mathbb R}^d$ and any map $f:U\rightarrow\mathbb{R}^d$ such that $F|_U=f$ for some holomorphic map $F:\mathbb{C}^d\rightarrow\mathbb{C}^d$, the map $f$ is affine in the form,
\[f(x)=Ax+b\]
with $A\in\mathcal{G}(w)$ and $b\in\mathbb{R}^d$ 
if and only if the composition operator $C_f:H_k\rightarrow H_{k|_{U^2}}$ is defined on the whole space $H_k$ and is bounded. 
Here, we recall that $P_n$ is the space of $d$-variable polynomials of degree smaller than or equal to $n$. 
\end{theorem}

We always regard the space of $d$-variable polynomials $\mathbb{C}[\xi_1,\dots,\xi_d]$ as a subspace of $L^2(w)$ as functions of $(\xi_1,\dots,\xi_d)$. We also fix an open set $U\subset \mathbb{R}^d$ and a map $f:U\rightarrow \mathbb{R}^d$ such that $F|_U=f$ for some $F=(F_1,\dots, F_d):\mathbb{C}^d\rightarrow \mathbb{C}^d$.

The following simple proposition shows that the information of $F$ is included in $K_f$ 
although $K_f$ is defined by the map $f$ initially defined only on $U$:
\begin{proposition}
Assume Assumption \ref{thm A: decay condition} in Theorem \ref{thm: affineness}. For any $z\in\mathbb{C}^d$, we have
\[K_f(\varphi(z))=\varphi(F(z)).\]
\end{proposition}

\begin{proof}
By Proposition \ref{prop: analyticity of feature maps}, both $K_f\circ\varphi$ and $\varphi\circ F$ 
are $L^2(w)$-valued holomorphic functions on $\mathbb{C}^d$. 
Since both holomorphic maps are identical on the open set $U\subset\mathbb{R}^d$ by definition, their values are the same on $\mathbb{C}^d$. 
\end{proof}

The following lemma is crucial for controlling the Jacobian matrix of $F$:
\begin{lemma}\label{lem:190515-4}
Assume Assumption \ref{thm A: decay condition} in Theorem \ref{thm: affineness} and $\mathcal{D}(C_f)=H_k$.
Then for any $d$-variable homogeneous polynomial $q\in\mathbb{C}[\xi_1,\dots,\xi_d]$, we have
\begin{align}
K_f(q{\bf e}_z)
=\left(\mathscr{S}\left[ \frac{i}{2\pi}J_F(z) \right]q + r\right){\bf e}_{F(z)},
\label{eq: Kfq}
\end{align}
where $r\in\mathbb{C}[\xi_1,\dots,\xi_d]$ is a polynomial of degree smaller than ${\rm deg}(q)$ without a constant term, and $\mathscr{S}[A]:\mathbb{C}[\xi_1,\dots,\xi_d]\rightarrow\mathbb{C}[\xi_1,\dots,\xi_d]$ is the symmetric product of a matrix $A = (a_{ij})$ of size $d$, namely, the linear map defined on $\mathbb{C}[\xi_1, \dots, \xi_d]$ via the correspondence 
$\xi_k \mapsto \sum\limits\limits_{m=1}^d a_{mk}\xi_m$.
\end{lemma}

\begin{proof}
It suffices to show that in the case of $q(\xi_1,\dots, \xi_d) =\xi_{i_1}\cdots \xi_{i_k}$ ($i_j \in \{1,\dots, d\}$), 
\begin{align}
    K_f(q {\bf e}_z) - \prod_{m=1}^k \left[\sum\limits_{j=1}^d \partial_{z_j} F_{i_m}(z) \cdot \frac{i\xi_j}{2\pi}\right]\cdot {\bf e}_{F(z)} = r_z e_{F(z)},\label{claim}
\end{align}
where $r_z$ is a polynomial of degree smaller than $k$ and its coefficients are entire functions with respect to the variable $z$.
We prove (\ref{claim}) by induction on $k$.
In the case of $k=1$, since $K_f$ is continuous, we see that 
\[K_f[\partial_{z_{i_1}}\varphi(z)]=\partial_{z_{i_1}}(K_f[\phi(z)]).\]
The left hand side is equal to $K_f[-2\pi i \xi_{i_1} {\bf e}_z]$, and the right hand side is equal to 
$\left(\sum\limits_{j=1}^d \partial_{z_j} F_{i_1}(z) \cdot \xi_j\right)\cdot {\bf e}_{F(z)}$.
Thus, we have (\ref{claim}).
Let $k>1$ and for $\varepsilon >0$, put
\[\psi_\varepsilon := i\frac{\varphi(z+\varepsilon e_{i_k})-\varphi(z)}{2\pi \varepsilon},\]
where $e_{i_k}$ is the $i_k$-th elementary vector in $\mathbb{C}^d$.
Then, by induction hypothesis, we have
\begin{align*}
    &K_f[\xi_{i_1}\cdots \xi_{i_{k-1}} \psi_\varepsilon]\\
    &=\prod_{m=1}^{k-1} \left[\sum\limits_{j=1}^d \partial_{z_j} F_{i_m}(z+\varepsilon e_{i_k}) \cdot \frac{i\xi_j}{2\pi}\right]\cdot \psi_\varepsilon \\
    &~+ \frac{i{\bf e}_{F(z)}}{2\pi \varepsilon}\left\{\prod_{m=1}^{k-1} \left[\sum\limits_{j=1}^d \partial_{z_j} F_{i_m}(z+\varepsilon e_{i_k}) \cdot \frac{i\xi_j}{2\pi}\right] - \prod_{m=1}^{k-1} \left[\sum\limits_{j=1}^d \partial_{z_j} F_{i_m}(z) \cdot \frac{i\xi_j}{2\pi}\right]\right\}\\
    &\qquad+ i\frac{s_{z+\varepsilon e_{i_k}}{\bf e}_{F(z+\varepsilon e_{i_k})} - s_{z} {\bf e}_{F(z)}}{2\pi\varepsilon}.
\end{align*}
where $s_z \in \mathbb{C}[\xi_1,\dots, \xi_d]$ is a polynomial of degree smaller than $k-1$ and their coefficients are entire with respect to $z$.
Then, if we take $\varepsilon$ to $0$, since $K_f$ is continuous, we see that there exists $r_z \in \mathbb{C}[\xi_1,\dots,\xi_d]$ of degree smaller than $k$ whose coefficients are entire with respect to $z$ such that
\[K_f[q {\bf e}_z] = \prod_{m=1}^{k} \left[\sum\limits_{j=1}^d \partial_{z_j} F_{i_m}(z) \cdot \frac{i\xi_j}{2\pi}\right] + r_z {\bf e}_{F(z)}.\]
Thus, by induction, we prove (\ref{claim}).

\end{proof}
For each $n>0$, we denote the space of homogeneous polynomials of degree $n$ by $\mathbb{C}[\xi_1,\dots,\xi_d]_n$. Then, we have the following corollary:

\begin{corollary}\label{lem:190515-5}
The situation is the same as that in Lemma \ref{lem:190515-4}. Write
\begin{align*}
Q_{n,z}
&=\{q\varphi_z\,:\,
q\in P_n\}
\subset L^2(w).
\end{align*}
Then, the following diagram is commutative{\rm:}
\[
\xymatrix@C=50pt{
Q_{n,z}\ar[r]^{K_f}\ar[d]^{{\rm proj.}}&Q_{n,F(z)}\ar[d]^{{\rm proj.}}\\
Q_{n,z}/Q_{n-1,z}\ar[r]^(0.45){[K_f]}&
Q_{n,f(z)}/Q_{n-1,F(z)}\\
P_n/P_{n-1} \ar[u]^\cong_{[m_{-2\pi i z}]} & P_n/P_{n-1} \ar[u]^\cong_{[m_{-2\pi i F(z)}]}\\
\mathbb{C}[\xi_1,\dots,\xi_d]_n \ar[u]^{\cong} \ar[r]^{\mathscr{S}^n[iJ_F(z)/2\pi]} & \mathbb{C}[\xi_1,\dots,\xi_d]_n \ar[u]^{\cong}.
}
\]
Here, ${\rm proj.}$ is 
the natural surjection to the quotient, $[\cdot]$ means the natural morphism induced by $(\cdot)$, 
and we define ${\mathscr{S}^n[iJ_F(z)/2\pi]}$ to be the restriction of ${\mathscr{S}[iJ_F(z)/2\pi]}$ to $\mathbb{C}[\xi_1,\dots,\xi_d]_n$.
\end{corollary}
\begin{proof}
This is clear from Lemma \ref{lem:190515-4}. 
\end{proof}
Regarding Corollary \ref{lem:190515-5}, we have the following lemma: 
\begin{lemma}
\label{inequality of operator norms}
The situation is the same as in Corollary \ref{lem:190515-5}. 
Then, we have the following inequality on the norm of the operators:
\begin{align*}
\|[K_f]\|&\le \|K_f\|,\\
\|[m_{-2\pi iz}]\|&\le \|m_{2 \pi {\rm Im}(z)}|_{P_n}\|,\\
\|[m_{-2\pi iF(z)}]^{-1}\| & \le \|m_{2 \pi {\rm Im}(F(z))}|_{P_n}^{-1} \|.
\end{align*}
Here, the topologies of the quotients space the above operators act on are induced from $L^2(w)$.
\end{lemma}
\begin{proof}
This lemma immediately follows the fact that $\|[*]\|$ is the same as the norms of the compositions of an inclusion and a projection for subspaces of $L^2(w)$ with $*$. 
\end{proof}

Next, we have the following key lemma:
\begin{lemma}\label{lem:190515-6}
Assume that Assumptions \ref{thm A: decay condition} and \ref{thm A: technical condition} in Theorem \ref{thm: affineness} hold. Then, the map $z \mapsto {\rm tr} J_F(z)$ is a constant function.
\end{lemma}

\begin{proof}
For each $n>0$, we denote by $\|\cdot\|_n$ the norm on $\mathbb{C}[\xi_1,\dots,\xi_d]_n$ induced from $P_n/P_{n-1}$ via the isomorphism
(see Corollary \ref{lem:190515-5}). 
Here, the norm of $P_n$ is the restriction of that of $H_k$, and the norm of $P_n/P_{n-1}$ is the quotient norm. Let $\alpha_z$ be an arbitrary eigenvalue of $J_F(z)$ that acts on $\mathbb{C}[\xi_1,\dots,\xi_d]_1$. Also, let $v \in \mathbb{C}[\xi_1,\dots,\xi_d]_1$ be its eigenvector. Then, we have
\[
\|\alpha_z{}^n v^n \|_n
=
\|\mathscr{S}^n[J_F(z)] v^n\|_n (2\pi)^{-n}.
\]
By Corollary \ref{lem:190515-5} and Lemma \ref{inequality of operator norms} (we use the notation in this corollary,) we have
\begin{align*}
   |\alpha_z|^n &\le (2\pi)^{-n}\|K_f\|\cdot \|[m_{-2\pi i z}]\| \cdot \big\|[m_{-2\pi i F(z)}]^{-1}\big\| \\
   &\le (2\pi)^{-n}\|K_f\|\cdot \|m_{2\pi  {\rm Im}(z)}|_{P_n}\| \cdot \big \|m_{2\pi  {\rm Im}(F(z))}|_{P_n}^{-1}\big\| 
\end{align*}
If we take the $n$-th root and then ``$\limsup\limits_{n \to \infty}$'', 
by combining this limit and Assumption \ref{thm A: technical condition}, 
there exists $B>0$ independent of $z$ such that 
\[
\sup\limits_z|\alpha_z|\le\frac{B^2}{2\pi}.
\]
Thus, any eigenvalue of $J_F(z)$ is bounded by a constant independent of $z$. In particular, the holomorphic function ${\rm tr}J_F$ is bounded on $\mathbb{C}^d$, and hence, ${\rm tr}J_F$ is constant by Liouville's theorem. 
\end{proof}

Now we prove Theorem \ref{thm: affineness}.
\begin{proof}
The ``only if'' part immediately follows from Proposition \ref{prop: boundedness for affine maps}. The ``if'' part is proved as follows: If $f$ is such a map, then, by Lemma \ref{lem:190515-6}, ${\rm tr}J_{A \circ f}(z)={\rm tr}(A J_F(z))$ is constant for any $A \in {\mathcal G}(w)$ as $C_A \circ C_f= C_{A\circ f}$ is bounded. By Assumption \ref{thm A: generator condition} in Theorem \ref{thm: affineness}: $\langle {\mathcal G}(w) \rangle_{{\mathbb R}}={\rm M}_d({\mathbb R})$, it follows that $J_F$ itself is independent of $z$. Thus, $f$ is an affine map. 
\end{proof}

\subsection{Analytic continuation}
\label{s3.1}
In this section, we prove that any map inducing bounded composition operators in $H_k$ has an analytic continuation:
\begin{theorem}
\label{thm: analytic continuation}
Let $w \in L^1 \cap L^\infty \setminus \{0\}$ be a non-negative function,
and let $k(x,y)=\widehat{w}(x-y)$.
We require that the function $u$ satisfies $({\rm DC})_{n,a}^d$ for some $a>0$. Then, for any open set $U\subset\mathbb{R}^d$ and any map $f:U\rightarrow\mathbb{R}^d$, there exists a holomorphic map $F:\mathfrak{X}_a^d\rightarrow\mathbb{C}^d$ such that $F|_U=f$ as long as $\mathcal{D}(C_f)=H_k$ and $C_f$ is bounded.
\end{theorem}
First, we give a simple lemma to prove analyticity of $f$ on $U$, as in Lemma \ref{lem:190515-7} below:
\begin{lemma}
\label{lem:190515-1}
Assume $u\in C^1$ and $|u(x)|\rightarrow 0$ as $|x|\rightarrow\infty$. 
Then, we have
\[\left\langle \{\nabla u(a)\}_{a \in {\mathbb R}^d}\right\rangle_{\mathbb{C}}
=\mathbb{C}^d. \]
\end{lemma}
\begin{proof}
Assume that ${\rm Span}(\{\nabla u(a)\}_{a \in {\mathbb R}^d})$ is a proper subset of ${\mathbb R}^d$. Then by a change of coordinates with a linear transformation, we may assume that
\[
{\rm Span}(\{\nabla u(a)\}_{a \in {\mathbb R}^d})
\subset \{x_d=0\}
\]
but it contradicts the fact that $u(x) \rightarrow 0$ as $|x| \rightarrow \infty$. 
\end{proof}
\begin{lemma}\label{lem:190515-7}
Let $w \in L^1 \cap L^\infty \setminus \{0\}$ be a non-negative function satisfying the condition $({\rm DC})_{n,a}^d$ for some $n>0$ and $a\ge0$. Let $k(x,y):=\widehat{w}(x-y)$.  If $\mathcal{D}(C_f)=H_k$ and $C_f$ is bounded, then $f$ is a $C^1$-function on $U$.
\end{lemma}
\begin{proof}
Put $u = \widehat{w}$.
By Lemma \ref{lem:190515-1}, we can find vectors $a_1,a_2,\ldots,a_d$ such that $\{\nabla u(a_j)\}_{j=1}^d$ spans ${\mathbb C}^d$.
Fix $b \in U$ arbitrarily.
It suffices to show that $f$ is $C^1$ at a neighborhood of $b$ in $U$.
Define
\begin{align*}
f_b(x)&=\big([\phi_k(f(b)-a_1)](x),\dots,[\phi_k(f(b)-a_d)](x)\big)\\[1pt]
&=(u(a_1+x-f(b)),
u(a_2+x-f(b)),\ldots,
u(a_d+x-f(b))).
\end{align*}
Then, 
\[
J_{f_b}(f(b))=
(\nabla u(a_1),\ldots,\nabla u(a_d)),
\]
so that $f_b$ induces a bijective $C^1$-map on the open ball $U_b$ centered at $f(b)$ into $\mathbb{R}^d$, and $f_b^{-1}$ is also a $C^1$-map on $f_b(U_b)$. Since ${\mathcal D}(C_f)=H_k$,
\[
[C_f f_b](x):=
\left\{C_f[\phi_k(f(b)-a_i)](x)\right\}_{i=1}^d
\]
is also a $C^1$-function defined on ${\mathbb R}^d$. Furthermore,
\begin{align*}
[C_f f_b](x)
&=\left\{C_f[\phi(f(b)-a_i](x)\right\}_{i=1}^d\\
&=\phi(f(b)-a_i)(f(x))\\
&=\left\{u(a_i+f(x)-f(b))\right\}_{i=1}^d\\
&=f_b \circ f(x).
\end{align*}
Consequently, $[C_f f_b](b)=f_b (f(b)) \in f_b(U_b)$. Since $f$ is continuous by Corollary \ref{cor: continuity of f}, we can find a neighborhood $V_b$ of $b$ such that $[C_f f_b](V_b) \subset f_b(U_b)$. Thus, on $x\in V_b$, we have $f(x)=f_b^{-1} \circ [C_f f_b](x)$. Therefore, $f$ is a $C^1$-function on a neighborhood of $b$. 
\end{proof}

Now let us prove Theorem \ref{thm: analytic continuation}. 
\begin{proof}
First, we claim that we may replace $w$ with $\tilde{w}(\xi) := (w(\xi) + w(-\xi))/2$.
In particular, this allows us to assume that $w$ is an even function, and thus, $-1\in\mathcal{G}(w)$.
Let $v := \widehat{\tilde{w}}$, and we define $\ell(x,y):=v(x-y)$.
We show the claim as follows: in fact, 
it is obvious that $v$ satisfies $({\rm DC})_{n,a}^d$.
We prove that the composition operator from $H_\ell$ to $H_{\ell|_{U^2}}$ is defined everywhere and bounded. 
We define a densely defined linear map $\widetilde{K}_f:L^2(\tilde{w})\rightarrow L^2(\tilde{w})$ with domain $\mathcal{D}(\widetilde{K}_f)=\left\langle \{{\bf e}_x\}_{x\in U}\right\rangle_\mathbb{C}$ by allocating ${\bf e}_{f(x)}$ to ${\bf e}_x$ (Here, we may define such a linear map since ${\bf e}_x$'s are linearly independent).
Let $u:= \widehat{w}$, and let $h=\sum\limits_{j=1}^ra_j{\bf e}_{x_j}\in\mathcal{D}(\widetilde{K}_f)$ ($a_j\in\mathbb{C}^d$ and $x_j\in U$ for $j=1,\dots, r$).
Then
\begin{align*}
\left\|\widetilde{K}_fh\right\|^2_{L^2(v)}
&=\sum\limits_{i,j=1}^r a_i\overline{a}_jv\left(f(x_i)-f(x_j)\right)\\
&=\sum\limits_{i,j=1}^r\frac{a_i\overline{a}_j+a_j\overline{a}_i}{2}u(f(x_i)-f(x_j))\\
&\le
\|K_f\|^2\sum\limits_{i,j=1}^r\frac{a_i\overline{a}_j+a_j\overline{a}_i}{2}u(x_i-x_j)\\
&=\|K_f\|^2\cdot\|h\|^2.
\end{align*}
Thus, we see that $\widetilde{K_f}$ is bounded and we can uniquely extend $\widetilde{K}_f$ as a bounded linear operator on $L^2(w)$.
We define
\[\widetilde{C}_f:H_\ell\overset{\Psi_{\tilde{w}}}{\cong}L^2({\tilde{w}})\overset{\widetilde{K}_f^*}{\rightarrow}L^2({\tilde{w}})\overset{\Psi_{\tilde{w}}^{-1}}{\cong}H_\ell=H_{\ell,U}\overset{r_{U}}{\cong}H_{\ell|_{U^2}}.\]
For any $h\in H_\ell$, we have $[\widetilde{C}_fh](x)=\langle \widetilde{C}_fh, \phi_\ell(x)\rangle_{H_k}=\langle h, \phi_\ell(f(x))\rangle_{H_k}=h(f(x))$. Therefore $\widetilde{C}_f$ is simply the composition operator from $H_\ell$ to $H_{\ell|_{U^2}}$, which is defined everywhere and bounded. 

By the above claim, we may assume that $w$ is an even function and $-1\in\mathcal{G}(u)$. 
Fix $y=(y_1,\dots,y_d)\in U$, and define the holomorphic map $\mathbf{F}:\mathfrak{X}_{a}^d\rightarrow L^2(w)\otimes L^2(w)$ by
\begin{align*}
\mathbf{F}(z)
&=\sum\limits\limits_{j=1}^d
\int_{y_j}^{z_j}
\left[\partial_{z_j}\varphi\otimes\varphi](z_1,\ldots,z_{j-1},w,y_{j+1},\dots,y_d)\right]
{\rm d}w,
\end{align*}
Let ${\rm m}: L^2(w)\otimes L^2(w)\rightarrow L^1(w)$ be the natural multiplication map, and let $\iota:\mathbb{C}^d\rightarrow \sum\limits\limits_{i=1}^d\mathbb{C}\xi_i\subset L^1(w)$ be the linear isomorphism defined by allocating $\xi_j$ to the vector $\mathbf{e}_j:=(0,\dots,0,\overset{j}{1},0,\dots,0)$. Then, we have the following proposition:
\begin{proposition}
Under the above notation, let $ U_0$ be the connected component of $U$ including $y$. Then, ${\rm m}\circ (K_{f}\otimes K_{-f})\circ \mathbf{F}:\mathfrak{X}_{a}^d\rightarrow L^1(w)$ is holomorphic and its image of $ U_0$ is contained in the finite-dimensional vector space $\sum\limits\limits_{i=1}^d\mathbb{C}\xi_i\subset L^1(w)$. Moreover, for any $y\in U_0$, we have 
\[\left(\iota^{-1}\circ {\rm m}\circ (K_{f}\otimes K_{-f})\circ\mathbf{F}\right)(x)+f(x_0)=f(x).\]
\end{proposition}
\begin{proof}
Since ${\rm m}$ and $(K_{f}\otimes K_{-f})$ are bounded linear operators and $\mathbf{F}$ is obviously holomorphic, the composition ${\rm m}\circ (K_{f}\otimes K_{-f})\circ \mathbf{F}$ is also holomorphic. Let $x\in U_0$. Thanks to Proposition \ref{prop: analyticity of feature maps} and Lemma \ref{lem:190515-4}, we see that
\[\left({\rm m}\circ (K_{f}\otimes K_{-f})\circ \mathbf{F}\right)(x)=\sum\limits_{j=1}^d(f_j(x)-f_j(x_0))\xi_j,\]
where $f:=(f_1,\dots,f_d)$.  
\end{proof}

Now, we complete the proof. Since Lemma \ref{lem: Hahn-Banach} below 
implies $$({\rm m}\circ\mathbf{F})(\mathfrak{X}_a^d)\subset\sum\limits\limits_{i=1}^d\mathbb{C}\xi_i,
$$ 
\[F:=f(x_0)+\left(\iota\circ {\rm m}\circ (K_{f}\otimes K_{-f})\circ \mathbf{F}\right):\mathfrak{X}_a^d\rightarrow\mathbb{C}^d\]
is well defined on $\mathfrak{X}_a^d$ and gives the analytic continuation of $f$. 
\end{proof}

\begin{lemma}
\label{lem: Hahn-Banach}
Let $U\subset\mathbb{R}^d$ be an open set, and let $X\subset\mathbb{C}^d$ be a connected open set containing $U$. Also, let $\mathbf{V}$ be a locally convex space over $\mathbb{C}$, and let $\gamma:X\rightarrow \mathbf{V}$ be a weakly holomorphic map. If $\gamma(U)$ is contained in a finite-dimensional subspace $V_0\subset \mathbf{V}$ over $\mathbb{C}$, then we have $\gamma(X) \subset V_0$.
\end{lemma}

\begin{proof}
Suppose that $\gamma(z_0)\notin V_0$ for some point $z_0 \in X$. Then, the Hahn-Banach theorem guarantees that there is a continuous linear functional $\lambda: \mathbf{V}\rightarrow \mathbb{C}$ such that $\lambda(\gamma(z_0))=1$ and $\lambda(V_0)=\{0\}$. Therefore, $\lambda\circ \gamma:X\rightarrow\mathbb{C}$ is a holomorphic function vanishing at $U$; thus, $\lambda\circ \gamma \equiv 0$ on $X$. This contradicts $\lambda\circ\gamma(z_0)=1$.  
\end{proof}

\section{Boundedness for special positive definite functions}
\label{section4}

In this section, we investigate the boundedness of composition operators on RKHSs 
associated to a spherical positive definite function
 and a convolution of real positive definite functions on $\mathbb{R}$.
Unless $w$ is compactly supported, Assumption \ref{main thm assumption 2}, 
that
is the existence of $B>0$ such that for any $y\in \mathbb{R}^d$,
\[ \limsup_{n\rightarrow \infty} \mathcal{E}_n^-(y; w), \limsup_{n\rightarrow\infty} \mathcal{E}^+_n(y; w) < B, \]
is the hardest condition to verify in our main theorem 
(Theorem \ref{main thm}).
In the case where $w$ is spherical or convolution of real positive definite functions, we may relate the left-hand-side 
to an asymptotic behavior of greatest zeros of certain orthogonal polynomials.
On the other hand, the asymptotic properties of greatest zeros of orthogonal polynomials are extensively studied 
in previous works (see, for example, \cite[Part 2]{Nev86}), and we can utilize various techniques developed there. 
Analysis illustrated in this section also provides a non-trivial consequence even in a one-dimensional case.
Indeed, our result 
is
still valid even if an RKHS contains an entire function of infinite order where previous frame work does not work.

\subsection{Orthogonal polynomials}
First, we briefly review basic properties of orthogonal polynomials.
For more details, see \cite{Nev86, Sze}.
Let $\mu$ be a Borel measure on $\mathbb{R}$.
Assume, for any integer $n \ge 0$, 
\begin{align}
    \int_{\mathbb{R}} |t|^{n} {\rm d}\mu(t) < \infty. 
\label{finite moment property}
\end{align} 
Then, we define (normalized) {\em orthogonal polynomials} by polynomials 
$$p_0(t; \mu), p_1(t; \mu), \ldots \in \mathbb{R}[t]$$ 
such that $p_n(t; \mu)$ is of degree $n$ and
\[\int_{\mathbb{R}} p_m(t; \mu) p_n(t; \mu) {\rm d}\mu(t) = \delta_{m,n}\]
for any $m, n \ge 0$. 
We note that each $p_n(t; \mu)$ is uniquely determined up to sign.
We denote by $X_n(\mu)$ the greatest zero of $p_n(t;\mu)$.

We denote by $\gamma_n(\mu)$ the leading coefficient of $p_n(t; \mu)$, and we define the Christoffel function $\lambda_n(s; \mu)$ by 
\begin{align}
\lambda_n(s; \mu) &:= \left( \sum\limits_{k=0}^{n-1} p_k(s;\mu)^2 \right)^{-1}.
\end{align}
We recall the following somewhat known properties:
\begin{proposition}
\label{prop: monotone property}
Let $\mu, \nu$ be
finite Borel measures on $\mathbb{R}$
satisfying $(\ref{finite moment property})$ such that $\nu- \mu \ge0$.

Then 
\begin{align}
    \gamma_n(\mu)^{-2} &\le \gamma_n(\nu)^{-2}\\
    \lambda(s;\mu) &\le \lambda(s;\nu).
\end{align}
\end{proposition}
\begin{proof}
We invoke the following formulas (see, for example, \cite[Theorem 3.1.2]{Sze}, and \cite[(4.1.1)]{Nev86}):
for any Borel measure satisfying (\ref{finite moment property}), we have
\begin{align*}
    \gamma_n(\mu)^{-2} &= \min_{p \in P_{n-1}} \int_{\mathbb{R}} \left( t^{n} + p(t) \right)^2 {\rm d}\mu(t),\\
    \lambda_n(s; \mu) &= \min_{\substack{p \in P_{n-1} \\ p(s)=1}} \int_{\mathbb{R}} p(t)^2 {\rm d}\mu(t).
\end{align*} 
\end{proof}
\subsection{Asymptotic  estimation of greatest zeros of orthogonal polynomials}
We denote by $L^p(\mathbb{R})$ the usual $L^p$-space with respect to the Lebesgue measure on the real line $\mathbb{R}$ for $p \in (0,\infty]$.
We denote by $\|\cdot\|_p$ the $L^p$-norm of $L^p(\mathbb{R})$.
Let $W \in L^1(\mathbb{R}) \cap L^\infty(\mathbb{R})$
be
such that $W(t) > 0$ and $W(t) = W(-t)$ for almost all $t \ge 0$.
Let 
\[\mu := W(t){\rm d}t.\]
Before getting into the main body of this section, 
we introduce some notation:  
For any $L \in \mathbb{R}$ and 
for any measurable function $Q: [0,\infty) \rightarrow \mathbb{R}$, we define 
\[Q_L(t) := Q(t) - Lt.\] 
We also define the 
``monotonic'' part and 
``oscillated'' part for $Q$ as follows:
\begin{align*}
    Q^{\rm m}(t) &:= \inf\big\{T \in \mathbb{R} : \left| \{ s \ge t : T \ge Q(s)\}\right| > 0 \big\},\\
    Q^{\rm o} &:= Q - Q^{\rm m}.
\end{align*}
We note that $Q^{\rm o}$ is non-negative, and $Q^{\rm m}$ is non-decreasing and upper semi-continuous (thus right continuous: $\lim\limits_{t\searrow s}Q^{\rm m}(t) = Q(s)$).  
Since $Q_1^{\rm m} + Q_2^{\rm m} \le (Q_1 + Q_2)^{\rm m}$, 
\begin{align}
    Q_1^{\rm o} + Q_2^{\rm o} \ge (Q_1 + Q_2)^{\rm o}. \label{ineq for oscillation}
\end{align}
For simplicity, we 
abbreviate $(Q_L)^{\rm m}$ and $(Q_L)^{\rm o}$ 
to $Q_L^{\rm m}$ and $Q_L^{\rm o}$.

We will prove the following theorem:
\begin{theorem}
\label{thm: estimation of greatest zeros}
Let $W$ be as above.
Suppose that there exists a function
$Q:[0, \infty) \rightarrow \mathbb{R}$
 such that $W(t) = {\rm e}^{-Q(|t|)}$. 
Let $L>0$ be a positive number and assume $t^n{\rm e}^{Lt}W(t) \in L^1(\mathbb{R})$ for all $n\ge0$.
For $\sigma \le L$, we define
\begin{align*}
    W^\sigma(t) &:= {\rm e}^{\sigma |t|}W(t),\\
    \mu^{\sigma} &:= W^\sigma(t){\rm d}t.
\end{align*}
Assume in addition the following two conditions:
    \begin{align}\lim\limits_{t\rightarrow \infty} \frac{Q_L^{\rm m}(t)}{t} = + \infty , \label{infty for QL}
    \end{align}
and there exists $B>0$ such that for any sufficiently large $t>0$,
\begin{align}
    &\int_0^{\pi/2} Q_L^{\rm o}(t \cos\theta) {\rm d}\theta < B + \log t.
    \label{bdd for QLos}
\end{align}
Then, $X_n(\mu^{\sigma})n^{-1}$ uniformly converges to $0$ over $\sigma \le L$.
\end{theorem}

The proof of Theorem
\ref{thm: estimation of greatest zeros}
is based on some auxiliary estimates.
The following lemma is essential:
\begin{lemma}
\label{lem: estimation of greatest zeros}
Under the same notation 
as Theorem \ref{thm: estimation of greatest zeros},
for any $\rho>0$, we define 
\begin{align*}
    \Xi_{\rho,\sigma} &:= \big\| |\cdot|^\rho W^\sigma \big\|_{\infty},\\
    \xi_{\rho,\sigma} &:= \sup\limits_{\varepsilon >0} \big[ \essinf\left\{t \ge 0: t^\rho W^\sigma(t) > \Xi_{\rho,\sigma} - \varepsilon \right\} \big]
\end{align*}
Then, for any $\sigma \le L$, we have 
\[ X_n(\mu^{\sigma}) \le \left(2 + \frac{2}{3\pi n}\exp\left(\frac{2}{\pi} \int_0^{\pi/2} Q_L^{\rm o}(\xi_{4n, \sigma} \cos\theta) {\rm d}\theta\right)\right) \xi_{4n ,L} \]
\end{lemma}
In order to prove Lemma \ref{lem: estimation of greatest zeros}, we improve results  
in \cite{Fre74-1}.
Before proving Lemma \ref{lem: estimation of greatest zeros}, we provide several 
fundamental inequalities.

For $\xi \ge 0$ and a function $W$ as in the beginning of this section, 
as Freud did in \cite{Fre74-1},
we define 
\begin{align}
    W_\xi(t) & := W(t) \cdot \mathbf{1}_{[-\xi, \xi]}(t),\\
    G_\xi(W) & := \exp\left\{ \frac{1}{2\pi} \int_0^{2\pi} \log\left[W(\xi \cos\theta)\right] {\rm d}\theta\right\},
\end{align}
and we denote $\mu_\xi := W_\xi(t) {\rm d}t$.
First, we prove an elementary inequality for holomorphic functions.
Lemma \ref{lem: mean value theorem} below
will substitute 
for the technique
used in \cite{Fre74-1},
where Freud used the
Hardy space $H_2$ over the unit disc.
\begin{lemma}
\label{lem: mean value theorem}
Let $U \subset \mathbb{C}$ be an open subset containing $\overline{\mathbb{D}}:=\{ z \in \mathbb{C} : |z| \le 1\}$.
Let $f: U \rightarrow \mathbb{C}$ be a holomorphic map such that $f(0) \neq 0$.
Then, 
\begin{align}
    |f(0)| \le \exp\left\{ \frac{1}{2\pi}\int_{0}^{2\pi} \log|f({\rm e}^{i\theta})| {\rm d}\theta \right\}. \label{exp log}
\end{align}
\end{lemma}
\begin{proof}
Let $\alpha_1,\dots, \alpha_n \in \overline{\mathbb{D}}$ be the zeros of $f$ contained in $\overline{\mathbb{D}}$.
Then, we may assume
that
$f$ takes the form:
\[f(z) = {\rm e}^{g(z)} \prod_{i=1}^n (z - \alpha_i)^{m_i}, \]
where 
$g$ is a holomorphic function in $\overline{\mathbb{D}}$ and $m_1,\dots, m_n >0$.
Thus, it suffices to prove the case of $f={\rm e}^g$ or $f= z- \alpha$ for some $\alpha \in \overline{\mathbb{D}} \setminus \{0\}$.
In the case of $f={\rm e}^g$, 
it immediately follows from the mean value theorem for harmonic function.  
In the case of $f(z) =z-\alpha$,
we employ the following
well-known fact
for
$|\alpha| \le 1$:
\[ \frac{1}{2\pi}\int_{0}^{2\pi} \log\left|{\rm e}^{i\theta} - \alpha\right| {\rm d}\theta = 0. \]
In this case, the inequality (\ref{exp log}) 
is equivalent to the validity of the inequality
for
$|\alpha| \le 1$. 
\end{proof}

Then, we have an estimation of $\gamma_n(\mu)$:
\begin{lemma}
For any $\xi > 0$
and $n \in {\mathbb N}$, 
we have
\begin{align}
\label{lem: inequality 1}
\gamma_n(\mu)^2 \le \gamma_n(\mu_\xi)^2 \le \frac{2}{\pi\xi} \left(\frac{2}{\xi}\right)^{2n} G_\xi(W).
\end{align}
\end{lemma}
\begin{proof}
The first inequality is a direct consequence of Proposition \ref{prop: monotone property}.
We prove the second inequality.
We define an entire function $f_n$ by
\[ f_n(z) := z^{n}p_n(\xi(z+z^{-1})/2 ; \mu_\xi)\]
keeping in mind that 
\[
z^{n}p_n(\xi(z+z^{-1})/2 ; \mu_\xi)=
\sum\limits_{k=0}^n
a_{n}z^n
\left(\frac{\xi}{2}(z+z^{-1})\right)^k,
\]
where we write $p_n(t; \mu_\xi) = \sum\limits_{k=0}^n a_k t^k$.
Since 
\[ 
\gamma_n(\mu_\xi)
=
\left(\frac{\xi}{2}\right)^n a_n, \quad
\frac{1}{2\pi}\int_{0}^{2\pi} \log|\sin\theta| {\rm d}\theta = -\log 2, \]
by using Lemma \ref{lem: mean value theorem}, and the Jensen inequality, we have
\begin{align*}
    \left(\frac{\xi}{2}\right)^{2n}\gamma_n(\mu_\xi)^2G_\xi(W) &= \big|f_n(0)^2\big|G_\xi(W)\\
    &\le 2\exp\left\{ \frac{1}{2\pi} \int_{0}^{2\pi}\log\left[\big|f_n({\rm e}^{i\theta})^2W(\xi\cos\theta)\sin\theta\big| \right]{\rm d}\theta \right\}\\
    &\le \frac{1}{\pi} \int_{0}^{2\pi} \big|f_n({\rm e}^{i\theta})^2\big|W(\xi\cos\theta)|\sin\theta| {\rm d}\theta. 
\end{align*}
As for the last integral, since $|f_n({\rm e}^{i\theta})| = |p_n(\xi\cos\theta; \mu_\xi)|$
from the definition
of $f_n$
and
$\mu_\xi=W_\xi(t){\rm d}t$, we have
\begin{align*}
    &\frac{1}{\pi} \int_{0}^{2\pi} \big|f_n({\rm e}^{i\theta})^2\big|W(\xi\cos\theta)|\sin\theta| {\rm d}\theta \\
    &= \frac{1}{\pi} \int_{0}^{2\pi} p_n(\cos\theta; \mu_\xi)^2 W(\xi\cos\theta)|\sin\theta| {\rm d}\theta \\
    &= \frac{4}{\pi} \int_{0}^{\pi/2} p_n(\cos\theta; \mu_\xi)^2 W(\xi\cos\theta)\sin\theta {\rm d}\theta \\
    &= \frac{2}{\xi\pi} \int_{-\xi}^{\xi} p_n(t; \mu_\xi)^2 {\rm d}\mu_\xi(t)\\
    &= \frac{2}{\xi\pi}.
\end{align*}
Here, 
we employed a formula $p_n(t; \mu_\xi) = (-1)^np_n(-t; \mu_\xi)$ as $W_\xi$ is an even function
for the second equality.
Therefore, we have
\[ \left(\frac{\xi}{2}\right)^{2n}\gamma_n(\mu_\xi)^2G_\xi(W) \le \frac{2}{\xi\pi},\]
which is equivalent to the second inequality of (\ref{lem: inequality 1}). 
\end{proof}

Then, we have an estimation of the Christoffel function as a corollary of this lemma:
\begin{corollary}
For any $\xi >0$ and $s \in \mathbb{R}$ such that $|s|>\xi>0$, we have
\begin{align}
    \lambda_n(s;\mu)^{-1}(s) \le \lambda_n(s;\mu_\xi)^{-1}(s) \le \frac{8}{3\pi\xi}G_\xi(W)^{-1}\left(\frac{2|s|}{\xi} \right)^{2n-2}
\end{align}
\end{corollary}
\begin{proof}
The proof is completely the same as \cite[Lemma 2]{Fre74-1},
where Freud 
factorized
$p_n(z;\mu_\xi)$
and
used
the fact that any zeros of $p_n(t; \mu_\xi)$ are contained in $[-\xi, \xi]$ and that
$\xi^2 \ge \xi^2-a^2$
for all $\xi,a \in {\mathbb R}$ with $|\xi| \ge |a|$.
\end{proof}

According to Chebychev's theorem
(see \cite[(17), (18)]{Fre74-1})
\[
X_n(\mu)
=
\sup\limits_{P \in P_{n-1}}
\left\{
\int_{-\infty}^{\infty}
x P(x)^2W(x){\rm d}x
\div
\int_{-\infty}^{\infty}
P(x)^2W(x){\rm d}x
\right\}.
\]
As Freud did
in \cite{Fre74-1},
we split the integral
\[
\int_{-\infty}^{\infty}
x P(x)^2W(x){\rm d}x
=
\left(
\int_{-A\xi}^{A\xi}
+
\int_{{\mathbb R} \setminus [-A\xi,A\xi]}
\right)
x P(x)^2W(x){\rm d}x
\]
and we can prove the following estimate:
\begin{lemma}
\label{prop: estimation for greatest zeros: general formula}
For any $\xi>0$ and $A\ge 1$, we have
\begin{align}
    X_n(\mu) \le A\xi + \frac{8}{3\pi} \left(\frac{2}{\xi}\right)^{2n-1} G_\xi(W)^{-1} \int_{A\xi}^\infty t^{2n-1} {\rm d}\mu(x).
\end{align}
\end{lemma}
\begin{proof}
The proof is completely the same as \cite[Theorem 2]{Fre74-1}.
\end{proof}
As a corollary of this lemma, we have the following estimation:
\begin{corollary}
Under the same notation 
as Lemma \ref{prop: estimation for greatest zeros: general formula}.
For any $\rho>0$, we define 
\begin{align*}
    \Xi_\rho &:=  \big\| |\cdot|^\rho W \big\|_\infty, \\
    \xi_\rho &:= \sup\limits_{\varepsilon >0} \big[ \essinf\left\{t \ge 0: t^\rho W(t) > \Xi_\rho  - \varepsilon \right\} \big]
\end{align*}
Then, we have
\[X_n(\mu) \le \left(2 + \frac{2\Xi_{4n}\xi_{4n}^{-4n}}{3\pi n G_{\xi_{4n}}(W)}\right)\xi_{4n}.\]
\end{corollary}
\label{cor: estimation of greatest zeros: special2}
\begin{proof}
Let $A=2$ and $\xi = \xi_{4n}$
in 
Lemma \ref{prop: estimation for greatest zeros: general formula}.
Since
\begin{align*}
    \int_{2\xi_{4n}}^\infty t^{2n-1} W(t) {\rm d}t  &\le \Xi_{4n} \int_{2\xi_{4n}}^\infty t^{-1-2n} {\rm d}t \\
    & \le \frac{1}{2n}\cdot \frac{\xi_{4n}^{2n}}{2^{2n}} \cdot \frac{\Xi_{4n}}{\xi_{4n}^{4n}},
\end{align*} 
we have the desired estimation. 
\end{proof}
%

Before proving Lemma \ref{lem: estimation of greatest zeros}, we prove a several simple lemmas:
\begin{lemma}\label{lem:210129-1} 
Under the same notation 
as Lemma \ref{lem: estimation of greatest zeros}, 
$\xi_{\rho,\sigma} \le \xi_{\rho,L}$
for any $\sigma \in [0,L]$.
\end{lemma}

\begin{proof}
Let $\varepsilon > 0$ be an arbitrary positive number.
Let $\Omega_{\varepsilon, \sigma} := \{ t \ge 0 : t^\rho W^\sigma(t) > \Xi_{\rho, \sigma} -\varepsilon \}$, and let $\eta_{\varepsilon, \sigma} := \essinf \Omega_{\varepsilon, \sigma}$.
Then, by definition, for almost all $s \in [0, \eta_{\varepsilon, \sigma})$ and almost all $t \in \Omega_{\varepsilon, \sigma}$, we have $t \ge \eta_{\varepsilon, \sigma}$ and  
\begin{align*}
    s^\rho W^\sigma(s) \le \Xi_{\rho, \sigma} - \varepsilon < t^\rho W^\sigma(t). 
\end{align*}
We write
\[ \delta(t) := \big( {\rm e}^{(L-\sigma)t} - {\rm e}^{(L-\sigma)\eta_{\varepsilon, \sigma}}\big)t^\rho W^\sigma(t).\]
By multiplying the both sides by ${\rm e}^{(L-\sigma)s}$, we have
\begin{align*}
    s^\rho W^L(s) &< t^\rho ({\rm e}^{(L-\sigma)s} - {\rm e}^{(L-\sigma)t} + {\rm e}^{(L-\sigma)t})W^\sigma(t) \\
    & \le t^\rho W^L(t) - \delta(t)\\
    & \le \Xi_{\rho, L} - \delta(t).
\end{align*}

Then, by definition of $\Omega_{\delta(t), L}$, we see that $\Omega_{\delta(t), L} \cap [0, \eta_{\varepsilon, \sigma}]$ is a null
set, thus, we have $\eta_{\varepsilon, \sigma} \le \eta_{\delta(t), L} \le \xi_{\rho, L}$.
Hence, we have $\xi_{\rho, \sigma} \le \xi_{\rho, L}$.
\end{proof}

We refer back to the
proof of Lemma \ref{lem: estimation of greatest zeros}.
\begin{proof}
Write $\xi = \xi_{4n,\sigma}$, and $\Xi = \Xi_{4n, \sigma}$.
By the definitions of $\xi$ and $\Xi$, there exists a non-negative sequence $\{t_m\}_{m\ge 1}$ such that 
$t_m \rightarrow \xi$ and $W^\sigma(t_m) \rightarrow \Xi \xi^{-4n}$ as $m\rightarrow \infty$.
Put $W_1(t)= {\rm e}^{(\sigma-L)|\cdot|}$, $W_2 = {\rm e}^{-Q_L^{\rm m}(|\cdot|)}$, and $W_3 := {\rm e}^{-Q_L^{\rm o}(|\cdot|)}$.
Observe that $W^\sigma = W_1 W_2W_3$ and that
\begin{equation}
    \limsup_{m\rightarrow \infty}\frac{W_1(t_m)}{G_\xi(W_1)} 
=1 \label{ineq W1}.
\end{equation}
So,
it suffices to show that
\begin{align}
    \limsup_{m\rightarrow \infty}\frac{W_2(t_m)}{G_\xi(W_2)} &\le 1 \label{ineq W2},\\
    \limsup_{m\rightarrow \infty}\frac{W_3(t_m)}{G_\xi(W_3)} &\le \exp\left\{ \frac{1}{2\pi} \int_0^{2\pi} Q_L^{\rm o}(\xi|\cos\theta|) {\rm d}\theta\right\}. \label{ineq W3}
\end{align}
As for (\ref{ineq W2}), since $Q_L^{\rm m}$ is decreasing, we see that 
\begin{align*}
    \frac{W_2(t_m)}{G_\xi(W_2)} &= \exp\left\{ \frac{1}{2\pi} \int_0^{2\pi} Q_L^{\rm m}(\xi|\cos\theta|) -  Q_L^{\rm m}(t_m) {\rm d}\theta\right\} \\
    & \le \exp\left\{ \frac{2}{\pi} \int_0^{\infty} \mathbf{1}_{[t_m, \max(t_m, \xi)]}(t) \big(Q_L^{\rm m}(t) -  Q_L^{\rm m}(t_m)\big) \frac{{\rm d}t}{\sqrt{\xi^2 - t^2}}\right\}
\end{align*}
Thus, by the Lebesgue convergence theorem, we have (\ref{ineq W2}).
The ineqauliy (\ref{ineq W3}) is immediate as $W_3(t_m)\le 1$ and the right hand side is just $G_\xi(W_3)^{-1}$.
\end{proof}
Now, we prove Theorem \ref{thm: estimation of greatest zeros}.
\begin{proof}
Thanks to Lemma \ref{lem: estimation of greatest zeros}, it suffices to prove that
\[\xi_{\rho,L}/\rho \rightarrow 0 \]
as $\rho \rightarrow \infty$.
Put $\xi_\rho := \xi_{\rho ,L}$, $\Xi_\rho := \Xi_{\rho, L}$ for simplicity.
Let $V(t):= {\rm e}^{-Q_L^{\rm m}}(t)$. 
\[V(t) =  \big\| W^L\mathbf{1}_{[t,\infty)} \big\|_\infty.\]
First, we claim that
\begin{align}
    \Xi_\rho = \sup\limits_{t \ge 0}t^\rho V(t) = \lim\limits_{t\nearrow \xi_\rho}t^\rho V(t). \label{eq Xi}
    %
\end{align}
In fact, 
we first prove $\Xi_\rho \ge \sup\limits_{t \ge 0}t^\rho V(t)$.
Since for any $t \ge 0$ and $\delta >0$,  by definition, the set $S := \{s \ge t :W^L(s) > V(t) - \delta\}$ has positive 
measure, we see that $\Xi_\rho \ge s^\rho W^L(s) \ge t^\rho V(t) - t^\rho \delta$ for almost all $s \in S$.
It implies the desired inequality.
As $\sup\limits_{t\ge0}t^\rho V(t) \ge \lim\limits_{t \nearrow \xi_\rho}t^\rho V(t)$ is obvious, it suffices to show that $\lim\limits_{t \nearrow \xi_\rho}t^\rho V(t) \ge \Xi_\rho$.
For $\delta > 0$, let 
\[\Omega_\delta := \big\{ t \ge 0 : t^\rho W^L(t) > \Xi_\rho - \delta \big\},\] and define $\eta_\delta := \essinf\Omega_\delta$.
Take arbitrary $\eta_1 > \eta_\delta$.
Then, we see that $(\eta_\delta, \eta_1) \cap \Omega_\delta$ has positive volume and is contained in 
\[\big\{ t \ge \eta_\delta :W(t)> (\Xi_\rho -\delta)\eta_1^{-\rho} \big\}.\]
Thus, by definition of $V(\eta_\delta)$, we have $V(\eta_\delta) \ge (\Xi_\rho - \delta)/\eta_1^\rho$.
Since $\eta_1$ is arbitrary and $V$ is right continuous, we see that $\Xi_\rho - \delta \le \eta_\delta^\rho V(\eta_\delta)$ for sufficiently small $\delta > 0$.
Since, $\eta_\delta \nearrow \xi_\rho$ as $\delta \rightarrow 0$, we have $\lim\limits_{t \nearrow \xi_\rho}t^\rho V(t) \ge \Xi_\rho$.

Let us assume $\xi_\rho \ge 1$.
By (\ref{eq Xi}), we have
\[\lim\limits_{s \nearrow \xi_\rho} s^\rho V(s) \ge t^\rho V(t)\]
for all $0 \le t \le \xi_\rho$,
or equivalently,
\[  \lim\limits_{s \nearrow \xi_\rho} Q_L^{\rm m}(s) - Q_L^{\rm m}(t) \le \rho \log(\xi_\rho/t). \]
Let $q>0$ be an arbitrarily positive number.
By multiplying the both sides by $qt^{q-1}$ and integrating them from $0$ to $\xi_\rho$, we have
\begin{align}
    \label{inequality 0129}
    \xi_\rho^q\lim\limits_{s \nearrow \xi_\rho} Q_L^{\rm m}(s) - q\int_0^{\xi_\rho} t^{q-1} Q_L^{\rm m}(t) {\rm d}t \le \rho \xi_\rho^q.
\end{align}
Since 
\begin{align*}
\lefteqn{
    \left| q\int_0^{\xi_\rho} t^{q-1}Q_L^{\rm m}(t) {\rm d}t - Q_L^{\rm m}(0) \right| 
    }\\
    &\le  \int_0^{1} \left|Q_L^{\rm m}(t^{1/q}) - Q_L^{\rm m}(0)\right| {\rm d}t+ q\int_1^{\xi_\rho} t^{q-1}|Q_L^{\rm m}(t)| {\rm d}t\\
    &\le \int_0^{1} \left|Q_L^{\rm m}(t^{1/q}) - Q_L^{\rm m}(0)\right| {\rm d}t + q\int_1^{\xi_\rho} |Q_L^{\rm m}(t)| {\rm d}t,
\end{align*}
we see that the second term of (\ref{inequality 0129}) goes to $Q_L^{\rm m}(0)$ when $q\rightarrow 0$ as $Q_L^{\rm m}$ is right continuous and bounded on $[0,1]$.
Thus, by taking limit $q\rightarrow 0$ on (\ref{inequality 0129}), we have
\[ \lim\limits_{s \nearrow \xi_\rho} Q_L^{\rm m}(s) - Q_L^{\rm m}(0) \le \rho.\]
Therefore, we have
\[\liminf_{\rho \rightarrow \infty} \lim\limits_{s\nearrow \xi_\rho} \frac{Q_L^{\rm m}(s)}{s} \le \liminf_{\rho \rightarrow \infty} \frac{\rho}{\xi_\rho}.\]
By (\ref{infty for QL}), we see that the left hand side is $+\infty$, thus so is the right hand side.
\end{proof}

\subsection{Boundedness for spherical positive definite functions}
Let $w$ be a spherical function on $\mathbb{R}^d$, and let $L>0$.
Take $W: [0, \infty) \rightarrow \mathbb{R}$ be a function such that $w(\xi) = W(|\xi|)$.
Assume $|\xi|^n{\rm e}^{2L|\xi|}w(\xi) \in L^1$ for all $n \ge 0$.
For $y \in \mathbb{R}^d$ with  $|y| \le  L$, we introduce 
\[\widetilde{\mathcal{E}}_n(y; w):= \sqrt[n]{ \sup\limits_{P \in P_n \setminus \{0\}}\frac{\|{\rm e}^{|y||\cdot|} P\|_{L^2(w)}}{\|P\|_{L^2(w)}}}.\]

First, we show a relation between the zeros of orthogonal polynomials and $\mathcal{E}^+_y(w)$.
\begin{lemma}
\label{lem: for condition B}
Let $w$, $W$ and $L$ be as above.
For $0\le \sigma \le L$, we define the measure $\nu_{\sigma}$ on $\mathbb{R}$ by
\[ \nu_{\sigma} = |t|^{d-1}{\rm e}^{2\sigma|t|}W(|t|) {\rm d}t .\]
Then, for $y \in \mathbb{R}^d$ with  $|y| \le L$, we have
\begin{align*}
\mathcal{E}^+_n(y;w) \le \widetilde{\mathcal{E}}_n(y;w) 
\le  \exp\left(\frac{1}{n}\int_{0}^{|y|} X_{n+2}(\nu_{t}) {\rm d}t \right).
\end{align*}
\end{lemma}
\begin{proof}
The first inequality is obvious in terms of the Cauchy-Schwarz inequality.
We prove the second inequality.
For any real-coefficients
one-variable polynomial $p \in \mathbb{R}[t]$, let
\begin{align*}
    \phi_p(s) := \int_{\mathbb{R}} p(t)^2 {\rm d}\nu_{s}(t).
\end{align*}
By the Cauchy-Schwarz inequality, we have
\begin{align*}
    \frac{\phi'_p(s)}{\phi_p(s)} &= 2\cdot \frac{\int_\mathbb{R} |t| p(t)^2 {\rm d}\nu_{s}(t)}{\int_\mathbb{R} p(t)^2 {\rm d}\nu_{s}(t)}\\
 \nonumber   &\le 2\cdot \left( \frac{\int_\mathbb{R} t^2 p(t)^2 {\rm d}\nu_{s}(t)}{\int_\mathbb{R} p(t)^2 {\rm d}\nu_{s}(t)} \right)^{1/2}.
\end{align*}
By Theorem 1 of \cite{Fre86}, we have
\begin{align}
     \frac{\phi'_p(s)}{\phi_p(s)} \le 2|X_{n+2}(\nu_{s})|. \label{ineq for key lem}
\end{align} 
Therefore, 
by integrating both sides of (\ref{ineq for key lem}), we have
\begin{align}
    \frac{\phi_p(s)}{\phi_p(0)} \le \exp\left( 2\int_{0}^{s} |X_{n+2}(t)| {\rm d}t\right). \label{formula of ratio}
\end{align}  
On the other hand, by a direct computation, we have
\begin{align*}
    \| {\rm e}^{|y||\cdot|} P\|_{L^2(w)}^2 &\le \int_{\mathbb{R}^d} P(\xi)^2 {\rm e}^{2|y|\cdot|\xi|} w(\xi) {\rm d}\xi,\\
    &=\int_{0}^\infty \left( \int_{S^{d-1}} P(ru)^2{\rm d}u\right) r^{d-1}{\rm e}^{2|y|r}W(r){\rm d}r,
\end{align*}
where $S^{d-1}$ is the unit sphere, and ${\rm d}u$ 
is a suitable invariant measure on the sphere $S^{d-1}$.
Put $h(r) := \int\limits_{S^{d-1}} P(ru)^2{\rm d}u$.
Since  $h$ is an even polynomial, we have $h(t) = h(|t|)$ for $t\in\mathbb{R}$.
Thus,
\[\| m_y P\|_{L^2(w)}^2 = \frac{1}{2}\int_\mathbb{R} h(r) {\rm d}\nu_{|y|}(r).\]
Since $h$ is a positive polynomial of degree at most $2n$, it is well
known that there exists
a finite collection of polynomials $h_1, \ldots, h_K$ of degree at most $n$ such that $h=h_1^2 + \dots + h_K^2$ (for example, see \cite{Ben17}).
Thus we have
\begin{align*}
    \frac{\| m_y P\|_{L^2(w)}^2}{\|P\|_{L^2(w)}^2} &\le \frac{\sum\limits_i \phi_{h_i}( |y|)}{\sum\limits_i \phi_{h_i}(0)} \\ 
    &\le \max_{i=1,\ldots,K} \frac{\phi_{h_i}( |y|)}{\phi_{h_i}(0)}.
\end{align*}
Therefore, by (\ref{formula of ratio}), we have the desired inequality.
\end{proof}
We prove the following theorem:
\begin{theorem}
\label{thm: criterion for spherical pdf}
Let $w \in L^1 \cap L^\infty \setminus \{0\}$ satisfying Assumption $\rm \ref{main thm assumption 1}$.
Assume there exists a measurable function $Q:[0,\infty) \rightarrow \mathbb{R}$ such that $w(\xi) = {\rm e}^{-Q(|\xi|)}$ and for all $L > 0$, there exists $B>0$ such that for any sufficiently large $t>0$,
\begin{align}
    &\int_0^{\pi/2} Q_L^{\rm o}(t \cos\theta) {\rm d}\theta < B + \log t.
    \label{bdd for QLos 2}
\end{align}
Then, the function $w$ further satisfies Assumptions
$\rm \ref{main thm assumption 2}$ and 
$\rm \ref{main thm assumption 4}$.
\end{theorem}
\begin{proof}
Since $w$ is a
spherical function, the set $\mathcal{G}(w)$ contains all the orthogonal matrices, thus, $\mathcal{G}(u)$ generates the space of matrices, namely, \ref{main thm assumption 4} holds.
Let us prove Assumption \ref{main thm assumption 2}.
First, we prove  
\[
\limsup\limits_{n\to\infty} \widetilde{\mathcal{E}}_n(y; w) \le  1.
\]
Thanks to Lemma \ref{lem: estimation of greatest zeros}, it suffices 
to prove that $R_L(t) := Q(t) - (d-1)\log t - Lt$ satisfies the conditions (\ref{infty for QL}) and (\ref{bdd for QLos 2}) for all $L > 0$.
The condition (\ref{infty for QL}) is immediate since $w$ satisfies Assumption \ref{main thm assumption 1}.
Regarding condition (\ref{bdd for QLos 2}), since 
\[ R_L^{\rm o}(t) \le Q_{L+d-1}^{\rm o}(t) + (d-1)(t - \log{t})\mathbf{1}_{[0,1]}(t),\]
we have
\begin{align*}
    &\int_{0}^{2\pi} R_L^{\rm o}(t\cos \theta) {\rm d}\theta \\
    &\le \int_{0}^{2\pi} Q_{L+d-1}^{\rm o}(t\cos \theta) {\rm d}\theta + (d-1)\int_0^{\min(1,t)} (s - \log{s})\frac{{\rm d}s}{\sqrt{t^2 -s^2}}.
\end{align*}
Since the second term is constant for sufficiently large $t$, we see that $R_L$ satisfies condition (\ref{bdd for QLos 2}) as $Q_{L+d-1}$ satisfies (\ref{bdd for QLos 2}).
Thus, 
$\limsup\limits_{n\to \infty} \mathcal{E}^+_n(y; w) \le 1$ 
by the left inequality of Lemma \ref{lem: for condition B}.
We prove $\limsup_{n} \mathcal{E}^-_n(y; w) \le 1$.
Let $w_0(\xi) := {\rm e}^{-2|y|\cdot|\xi|}w(\xi)$.
Let $P \in P_n$.
By the Cauchy--Schwarz inequality, we have
\[\|m_y P\|_{L^2(w)} \ge \|P {\rm e}^{-|y||\cdot|}\|_{L^2(w)} = \|P\|_{L^2(w_0)}.  \]
Then, we have
\begin{align*}
    \frac{\|P\|_{L^2(w)}}{\|m_y P\|_{L^2(w)}} \le \frac{\|{\rm e}^{|y||\cdot|}P\|_{L^2(w_0)}}{\|P\|_{L^2(w_0)}} \le \widetilde{\mathcal{E}}_n(y; w_0).
\end{align*}
Thus, we may use Lemma \ref{lem: for condition B} and Theorem \ref{thm: estimation of greatest zeros} as in the above argument, and we obtain
\[\limsup_{n\rightarrow \infty} \mathcal{E}^-_n(y; w)\le \limsup_{n\rightarrow\infty} \widetilde{\mathcal{E}}_n(y; w_0) \le 1.\]
Therefore, Assumptions \ref{main thm assumption 1}, \ref{main thm assumption 2}, and \ref{main thm assumption 4} hold.
\end{proof}
As a result, we obtain a simple sufficient condition for $w$ 
so as to satisfy Assumption \ref{main thm assumption 2}:
\begin{corollary}
\label{cor: super easy criterion}
Let $w \in L^1 \cap L^\infty \setminus \{0\}$ be 
a nonnegative spherical function.
Assume that there exists
a locally $L^1$-function $Q:[0,\infty) \rightarrow \mathbb{R}$ such that $w(\xi) = {\rm e}^{-Q(|\xi|)}$.
We further assume that there exists $c \ge 0$ such that  $Q(t) + ct$ is non-decreasing for sufficiently large $t \ge 0 $ and
that $Q(t+R)-Q(t) \rightarrow \infty$ as $t\rightarrow \infty$ for some $R>0$.
Then, the function $w$ satisfies Assumptions
$\rm \ref{main thm assumption 1}$,
$\rm \ref{main thm assumption 2}$, and 
$\rm \ref{main thm assumption 4}$.
\end{corollary}
\begin{proof}
We deduce from the assumtions of the corollary
that $Q$ is bounded any interval $[a,b]$ as long as $b>a \gg 1$.
Assumption \ref{main thm assumption 1} immediately follows from the condition that $Q(t+R) - Q(t) \rightarrow \infty$ as $t \rightarrow \infty$.
We will prove that for any $L>0$,  $Q_L^{\rm o}(t)$ is bounded for any sufficiently large $t>0$.
We easily see that the boundedness of $Q_L^{\rm o}$ implies the condition (\ref{bdd for QLos 2}) in Theorem \ref{thm: estimation of greatest zeros}, and thus, $w$ satisfies the condition \ref{main thm assumption 2} and \ref{main thm assumption 4} by Theorem \ref{thm: criterion for spherical pdf}.

First, we claim that we may assume $Q$ is non-decreasing.
In fact, Let $R'>0$ be a positive number such that $Q_{-c}(t) = Q(t) + ct$ is non-decreasing for $t \ge R'$.
We define a non-decreasing function $\tilde{Q}$ by 
\[\tilde{Q}(t) := \big(Q_{-c}(t) -Q_{-c}(R') \big)\mathbf{1}_{[R',\infty)}.\]
Note that we immediately see that $\tilde{Q}(t+R) - \tilde{Q}(t) \rightarrow \infty$ as $t \rightarrow \infty$.
Then, we have
\begin{align*}
    Q_L &= Q_{-c}\mathbf{1}_{[0,R')} + Q_{-c}(R')\mathbf{1}_{[R',\infty)} + \tilde{Q}_{L+c},\\
    \left(Q_{-c}\mathbf{1}_{[0,R')}\right)^{\rm o}
    &\le \left(|Q| +  |Q^{\rm m}(0)| + cR'\right)\mathbf{1}_{[0,R')}.
\end{align*}
Thus, we have
\[ Q_L^{\rm o} \le  \left(|Q| +  |Q^{\rm m}(0)| + cR'\right)\mathbf{1}_{[0,R')} + \tilde{Q}_{L+c}.\]
Since the first term does not affect the condition (\ref{bdd for QLos 2}), we may replace $Q(t)$ with the non-decreasing function $\tilde{Q}$ to prove the condition (\ref{bdd for QLos 2}).

Now, we assume $Q$ is non-decreasing.
Fix an arbitrary sufficiently large number  $s \ge 0$ satisfying $Q(t+R) - Q(t) > LR$ for any $t \ge s$.
Since for any $t \in [s, s+R)$, we have $Q_L(t+nR) - Q_L(t) >0$ for all positive integer $n\ge1$.
Thus, for any arbitrary $\varepsilon > 0$ there exists $t_\varepsilon \in [s, s+R)$ such that  $Q_L^{\rm m}(s) + \varepsilon  > Q_L(t_\varepsilon)$.
Then, by definition of $Q_L^{\rm o}$
and by the fact that
$Q$ is non-decreasing, we see that
\[ Q_L^{\rm o}(s) - \varepsilon \le Q_L(s) - Q_L(t_\varepsilon) \le Q(s) - Q(t_\varepsilon) +LR \le L R .\]
Since  $\varepsilon$ is arbitrary, we have $Q_L^{\rm o}(s) \le LR$, which means that for any sufficently large $s>0$, $Q_L^{\rm o}(s)$ bounded.
\end{proof}

\subsection{Boundedness for positive definite functions of tensor products of even functions}
\label{sec: tensor prod}
In this subsection, we discuss the case where the non-negative function $w$ is a tensor product of even functions on $\mathbb{R}$.
First, we prove the following lemma:
\begin{lemma}
\label{lem: for tensor prod}
Let $w \in L^1 \cap L^\infty \setminus \{0\}$ be a nonnegative measurable function.
Assume there exists $w_1,\dots, w_d : \mathbb{R} \rightarrow \mathbb{R}$ such that each $w_i$ satisfies Assumption \ref{main thm assumption 1} and $w(\xi_1,\dots, \xi_d) = w_1(\xi_1)\cdots w_d(\xi_d)$.
If each $w_i$ satisfies Assumption \ref{main thm assumption 2}, namely, there exists $B_i > 0$ such that for all $y \in \mathbb{R}$, 
\[
    \limsup_{n\rightarrow \infty}\mathcal{E}^{\pm}_n(y; w_i) < B_i,
\]
then $w$ also satisfies Assumption \ref{main thm assumption 2}.
\end{lemma}
\begin{proof}
Fix $y = (y_i)_{i=1}^d \in \mathbb{R}$.
Then, for any $P \in P_n$, we see that
\begin{align*}
\frac{\|m_y P\|_{L^2(w)}^2}{\|P\|_{L^2(w)}^2} = \prod_{i=1}^d  \frac{\int_{\mathbb{R}} {\rm e}^{y_i\xi}P_i(\xi) w_i(\xi){\rm d}\xi}{\int_{\mathbb{R}} P_i(\xi) w_i(\xi){\rm d}\xi},
\end{align*}
where $P_i$ is a one-variable polynomial of degree at most $2n$.
\begin{align*}
P_i(\xi) := &\int_{\mathbb{R}^{d-1}} |P(\xi_1,\dots,\xi_{i-1},\xi,\xi_{i+1}, \dots,\xi_d)|^2 w_1(\xi_1)\cdots w_{i-1}(\xi_{i-1})\\
&~~~~~~\times {\rm e}^{y_{i+1}\xi_{i+1}}w_{i+1}(\xi_{i+1}) \cdots {\rm e}^{y_{d}\xi_{d}}w_d(\xi_d) {\rm d}\xi_1\cdots {\rm d}\xi_{i-1}{\rm d}\xi_{i+1}\cdots {\rm d}\xi_{d}
\end{align*}
Since each $P_i$ is a positive polynomial, by the Hilbert's 17th problem (see, for example \cite{Ben17}), there exists a finite collection of one-variable polynomials $Q_{i,1},\dots, Q_{i, K_i}$ of degree at most $n$ such that $P_i = Q_{i,1}^2 + \dots + Q_{i, K_i}^2$.
Thus, we have
\[
\frac{\int_{\mathbb{R}} {\rm e}^{y_i\xi}P_i(\xi) w_i(\xi){\rm d}\xi}{\int_{\mathbb{R}} P_i(\xi) w_i(\xi){\rm d}\xi}
\le \max_{j=1,\dots, K_i} \frac{\int_{\mathbb{R}} {\rm e}^{y_i\xi}Q_{i,j}(\xi)^2 w_i(\xi){\rm d}\xi}{\int_{\mathbb{R}} Q_{i,j}(\xi)^2 w_i(\xi){\rm d}\xi}
\le \mathcal{E}^+_n(y_i; w_i)^{2n}.
\]
Therefore, we see that 
\[\limsup_n\mathcal{E}_n^+(y, w) \le B_1\cdots B_d.\]
We also obtain 
\[\limsup_n\mathcal{E}_n^-(y, w) \le B_1\cdots B_d.\]
in the same manner as above.
\end{proof}
Then, we obtain a similar theorem to Theorem \ref{thm: criterion for spherical pdf}:
\begin{theorem}
\label{thm: criterion for tensor product pdf}
Let $w \in L^1 \cap L^\infty \setminus \{0\}$ satisfying Assumption $\rm \ref{main thm assumption 1}$.
Assume there exists a measurable function $Q_1,\dots, Q_d:[0,\infty) \rightarrow \mathbb{R}$ such that 
\[w(\xi_1, \dots, \xi_d) = \prod_{i=1}^d{\rm e}^{-Q_i(|\xi_i|)},\]
for all $i,j \in \{1,\dots, d\}$, $|Q_i - Q_j|$ is bounded, 
and for each $i=1,\dots, d$ and any $L>0$, there exists $B_i>0$ such that for any sufficiently large $t>0$,
\begin{align}
    &\int_0^{\pi/2} (Q_i)_L^{\rm o}(t \cos\theta) {\rm d}\theta < B_i + \log t.
\end{align}
Then, the function $w$ further satisfies Assumptions
$\rm \ref{main thm assumption 2}$ and 
$\rm \ref{main thm assumption 4}$.
\end{theorem}
\begin{proof}
By Lemma \ref{lem: for tensor prod} with Theorem \ref{thm: criterion for spherical pdf}, it suffices to show that $w$ satisfies Assumption \ref{main thm assumption 4}.
The boundedness of $|Q_i - Q_j|$ implies $\mathcal{G}(w)$ contains all the symmetric group $\mathfrak{S}_d \subset {\rm GL}_d(\mathbb{R}^d)$.
Since each $Q_i$ is an even function, $\mathcal{G}(w)$ also contains $C_2^d \subset {\rm GL}_d(\mathbb{R}^d)$, the group of diagonal matrices
$A$ satisfying $A^2=I$.
Since the group generated by $\mathfrak{S}_d$ and $C_2^d$ generates ${\rm M}_d(\mathbb{R})$ as a linear space over $\mathbb{R}$, we obtain \ref{main thm assumption 4}.
\end{proof}
We also have a similar corollary to Corollary \ref{cor: super easy criterion}:
\begin{corollary}
\label{cor: super easy criterion 2}
Let $w \in L^1 \cap L^\infty \setminus \{0\}$ satisfying Assumption $\rm \ref{main thm assumption 1}$.
Assume there exists a measurable function $Q_1,\dots, Q_d:[0,\infty) \rightarrow \mathbb{R}$ such that 
\[w(\xi_1, \dots, \xi_d) = \prod_{i=1}^d{\rm e}^{-Q_i(|\xi_i|)},\]
for all $i,j \in \{1,\dots, d\}$, $|Q_i - Q_j|$ is bounded.
We further assume that for each $i = 1, \dots, d$, there exists $c_i$ such that $Q_i(t) + c_it$ is non-decreasing for sufficiently large $t \ge 0$ and $Q_i(t+R_i) - Q(t) \rightarrow \infty$ as $t\rightarrow \infty$ for some $R_i>0$.
Then, the function $w$ satisfies Assumption \ref{main thm assumption 1}, \ref{main thm assumption 2}, and \ref{main thm assumption 4}.
\end{corollary}
\begin{proof}
By Lemma \ref{lem: for tensor prod} with Corollary \ref{cor: super easy criterion}, it suffices to show that $w$ satisfies Assumption \ref{main thm assumption 4}, but its proof is the same as in that of Theorem \ref{thm: criterion for tensor product pdf}.
\end{proof}
For example, $w(\xi) = {\rm e}^{-|\xi|_p^p}$ ($p>1$) satisfies the condition in Corollary \ref{cor: super easy criterion 2}, where $|\xi|_p:= (|\xi_1|^p + \dots + |\xi_d|^p)^{1/p}$ for $\xi = (\xi_i)_{i=1}^d \in \mathbb{R}^d$.

\subsection{An example of entire functions of infinite order}
\label{subsec: remark}
Even in one dimensional case, our result contains an essentially new contribution.
In the case of $d=1$ and $w = \mathbf{1}_{[-1/2, 1/2]}$ , there are several works treated the boundedness of composition operators in RKHS \cite{CCG, CC93}.
Their method based on finiteness of the order of the entire function $\widehat{w}$.
As the RKHS associated to the positive definite function $\widehat{w}$ is composed of entire functions of order $1$, they \cite{CCG, CC93} directly apply P\'olya's theorem \cite{Po26} with some careful analysis, and deduce affiness of original maps inducing bounded composition operators.
We may apply this method  without using ours discribed in this paper if the RKHS is composed of entire functions of finite order.
However, there exists an example of RKHSs containing entire functions of infinite order, but only affine maps can induce bounded composition operators on the RKHS based on our framework.

Let us explain the example.
We define
\[w(\xi) = \sum\limits_{n=-\infty}^\infty \frac{\mathbf{1}_{[-1/2+n, 1/2 +n)}(\xi)}{|n|!}.\]
Then, we easily see that
\begin{align*}
    \widehat{w}(z) &= \frac{\sin(\pi z)}{\pi z}\sum\limits_{n = -\infty}^\infty \frac{{\rm e}^{2\pi i nz}}{|n|!} \\
    &=\big({\rm e}^{{\rm e}^{2\pi i z}}+{\rm e}^{{\rm e}^{-2\pi i z}} -1 \big)\cdot \frac{\sin(\pi z)}{\pi z}
\end{align*}
The entire function $\widehat{w}(z)$ is of infinite order since $|\widehat{w}(iy)| = O({\rm e}^{{\rm e}^{2\pi y}})$.
Let $Q := -\log w$.
Then, we immediately see that $Q$ is non-decreasing and $Q(t+1) - Q(t) \rightarrow \infty$ as $t\rightarrow \infty$ since $Q(t+1) - Q(t) = \log (n+1)$ for $t\in [-1/2 + n , 1/2 +n)$.
Thus, thanks to Corollary \ref{cor: super easy criterion} and Theorem \ref{main thm}, if a composition operator on the RKHS associated with the positive definite function $\widehat{w}$ is bounded, then the original map is an affine map.



\newpage


\end{document}